%% file: ms.tex
\documentclass[reqno]{amsart}

\pdfoutput=1  
\usepackage{hhline} 
\usepackage{xcolor} 
\usepackage{caption} 
\usepackage{subcaption} 
\usepackage{amssymb} 
\usepackage{graphicx} 

\usepackage[a4paper,top=3cm,bottom=2cm,left=3cm,right=3cm,marginparwidth=1.75cm]{geometry}

\usepackage{amsthm} 
\theoremstyle{theorem}
\newtheorem{theorem}{Theorem}
\theoremstyle{corollary}
\newtheorem{corollary}{Corollary}
\theoremstyle{lemma}
\newtheorem{lemma}{Lemma}
\theoremstyle{remark}
\newtheorem*{remark}{Remark}
\newtheorem*{notation}{Notation}
\theoremstyle{definition} 
\newtheorem{definition}{Definition} 
\theoremstyle{proposition} 
\newtheorem{proposition}{Proposition} 

\begin{document}

\title{Hyperoperations in Exponential Fields}

\author{Juan D.~Jaramillo}

\address{
Universidad del Valle Sede Tuluá, Villa Campestre, Calle 43 No.~43-33, Tuluá, Colombia}

\email{juan.jaramillo.salazar@correounivalle.edu.co}

\begin{abstract}
New sequences of hyperoperations \cite{BE15,ACK28,GO47,TAR69} are presented together with their algebraic properties.  The commutative hyperoperations reported by Bennett \cite{BE15} are defined as a sequence of monoids.  After identifying the semirings along the sequence, the corresponding fields are constructed via inverse completion.
\end{abstract}

\maketitle

\section{Introduction}
\input{0-intro}

\section{Monoids and Semirings}\label{sec:semirings}
\input{1-semirings}

\section{New Sequences of Hyperoperations}\label{sec:newseq}
\input{2-new-seq}

\section{Representation of Natural Numbers}\label{sec:repr}
\input{3-rep-numb}

\section{Groups and Fields}\label{sec:fields}
\input{4-fields}

\subsection{Construction of abelian groups from commutative monoids}
\input{5-fields}

   \subsection{Construction of integral domains from abelian groups} 
   \input{6-fields}

   \subsection{Construction of fields of quotients from integral domains}
   \input{7-fields}

   \subsection{Construction of real numbers from fields of quotients}
   \input{8-fields}

\section{Embedding in the real rumbers}\label{sec:emb}
\input{9-embedding}

\medskip

\input{bib}
\end{document}

%% file: 0-intro.tex
Hyperoperations are infinite sequences of binary operations extending recursively the definition of the ordinary addition and multiplication \cite{BE15,ACK28,GO47,TAR69}.  The earliest known report on hyperoperations is due to Bennet \cite{BE15} in 1914, where commutative hyperoperations are defined starting from the multiplicative identity of the exponential function.  More than a decade later, Ackermann \cite{ACK28} introduces a sequence of hyperoperations that extends the definition of the ordinary addition, multiplication and exponentiation.  The latter sequence is here referred as regular hyperoperations, noting that they are noncommutative beyond multiplication.  A numerical system based on regular hyperoperations is introduced by Goodstein \cite{GO47} in 1947: the complete hereditary representation of nonnegative integers.  He also coined the names of tetration, pentation, etc.~to denote the regular hyperoperations after exponentiation, and included the \textit{successor} as the primitive regular hyperoperation.
 
 This study presents an algebraic characterization of Bennett's sequence of commutative hyperoperations and a recursive procedure to create new sequences.  The general goal is to provide a more solid ground to the study of hyperoperations.  The content is relevant for the study of commutative fields.  Possible applications could be found in computer arithmetic, where research remains active in the search for alternatives to the IEEE floating-point to overcome numeric overflow/underflow \cite{SA07,LI18}.  In particular, it provides a formal background to proposals involving generalized exponentiation such as the Elias $\omega$ code \cite{EL75} and the level-index codes \cite{CLE84,CLE88}.  It is also of interest to the subfield of weighted automata in computer science, where semirings play a central role \cite{SA09,DR09}. 
  
 This paper is organized as follows:  The section \ref{sec:semirings} reviews the original sequence of commutative hyperoperations and proves that consecutive hyperoperations form a commutative semiring.  The section \ref{sec:newseq} constructs new sequences of hyperoperations from the insight that each commutative hyperoperation is the initial object of a new sequence of hyperopertations.  In the section \ref{sec:repr} an algorithm to represent nonnegative numbers using commutative hyperoperations is presented.  In the section \ref{sec:fields} the method of inverse completion \cite{WA65} is used to construct a quotient field and its Cauchy completion for each of the former commutative semirings.  The section \ref{sec:emb} closes the paper embedding the former structures in the usual real numbers and makes remarks pointing to future research.

%% file: 1-semirings.tex
\begin{notation}
Denote the set of natural numbers from the Peano axioms as $\mathbb{N}=\{0,1,2,3,...\}$,  together with its successor function $S:\mathbb{N}\rightarrow\mathbb{N}\setminus\{0\}$ and its inverse function $S^{-1}$.
\end{notation}
\begin{definition}\label{defn:hyp}
 The sequence of regular hyperoperations generated from the succesor function in $\mathbb{N}$ is the sequence of binary operations $H_{n}:(\mathbb{N})^{2}\rightarrow\mathbb{N}$ for $n\in\mathbb{N}$, defined recursively as
 \begin{equation}
 H_{n}(a,b)=\left\{\begin{array}{ll}
 S(b),&\text{if}\ n=0;\\ 
 a,&\text{if}\ n=1\ \text{and}\ b=0;\\
 0,&\text{if}\ n=2\ \text{and}\ b=0;\\
 1,&\text{if}\ n\geq 3\ \text{and}\ b=0;\\
 H_{n-1}(a,H_{n}(a,S^{-1}(b))),&\text{otherwise}.
 \end{array}\right.
 \end{equation}
\end{definition}
\begin{proposition}\label{prop:Hsemiring}
The tuple $(\mathbb{N},H_{1},H_{2})$ is a commutative semiring.
\end{proposition}
\begin{proof}
Let $a,b,c\in\mathbb{N}$.
\begin{enumerate}
\item\textit{Closure of addition.}  By definition $H_{1}(a,0)=a\in\mathbb{N}$.  For $b>0$ and noting that $H_{0}$ is closed in $\mathbb{N}$, one obtains 
\begin{subequations}
\begin{align}
H_{1}(a,b)&=H_{0}(a,H_{1}(a,b-1))\\
&=H_{0}(a,H_{0}(a,H_{1}(a,b-2)))\\
&\vdots\nonumber\\ 
&=\underbrace{H_{0}(a,H_{0}(a,\ldots H_{0}(a,H_{1}(a,0))\ldots))}_{b\ \text{copies of}\ H_{0}}\\
&=\underbrace{H_{0}(a,H_{0}(a,\ldots H_{0}(a,H_{0}(a,a))\ldots))}_{b\ \text{copies of}\ H_{0}}\in\mathbb{N}.
\end{align}
\end{subequations}
\item\textit{Additive identity.} 
By definition $H_{1}(a,0)=a$.  On the other hand,
\begin{subequations}
\begin{align}
H_{1}(0,a)&=\underbrace{H_{0}(0,H_{0}(0,\ldots H_{0}(0,H_{0}(0,0))\ldots))}_{a\ \text{copies of}\ H_{0}}\\
&=\underbrace{H_{0}(0,H_{0}(0,\ldots H_{0}(0,H_{0}(0,1))\ldots))}_{a-1\ \text{copies of}\ H_{0}}\\
&=\underbrace{H_{0}(0,H_{0}(0,\ldots H_{0}(0,H_{0}(0,2))\ldots))}_{a-2\ \text{copies of}\ H_{0}}\\
&\vdots\nonumber\\
&=H_{0}(0,a-1)\\
&=a.
\end{align}
\end{subequations}
\item\textit{Commutativity of addition.}   Let $a<b$, then
\begin{subequations}
\begin{align}
H_{1}(a,b)&=\underbrace{H_{0}(a,H_{0}(a,\ldots H_{0}(a,H_{0}(a,a))\ldots))}_{b\ \text{copies of}\ H_{0}}\\
&=\underbrace{H_{0}(a,H_{0}(a,\ldots H_{0}(a,H_{0}(a,H_{1}(1,a))\ldots)))}_{b-1\ \text{copies of}\ H_{0}}\\
&=\underbrace{H_{0}(a,H_{0}(a,\ldots H_{0}(a,H_{1}(1,H_{1}(1,a)))\ldots))}_{b-2\ \text{copies of}\ H_{0}}\\
&\vdots\nonumber\\
&=\underbrace{H_{0}(a,H_{0}(a,\ldots H_{0}(a,H_{0}(a,H_{0}(a,b-1)))\ldots))}_{a+1\ \text{copies of}\ H_{0}}\\
&=\underbrace{H_{0}(a,H_{0}(a,\ldots H_{0}(a,H_{0}(a,b))\ldots))}_{a\ \text{copies of}\ H_{0}}\\
&=H_{1}(b,a).
\end{align}
\end{subequations}
\item\textit{Associativty of addition.}
\begin{subequations}
\begin{align}
H_{1}(a,H_{1}(b,c))&=\underbrace{H_{0}(a,H_{0}(a,\ldots H_{0}(a,H_{0}(a,H_{0}(a,a)))\ldots))}_{H_{1}(b,c)\ \text{copies of}\ H_{0}}\\
&=\underbrace{H_{0}(a,H_{0}(a,\ldots H_{0}(a,H_{0}(a,H_{1}(1,a))\ldots))}_{H_{1}(b,c)-1\ \text{copies of}\ a}\\
&\vdots\nonumber\\
&=\underbrace{H_{0}(a,H_{0}(a,\ldots H_{0}(a,H_{0}(a,H_{1}(a,b)-1))\ldots))}_{c+1\ \text{copies of}\ H_{0}}\\
&=\underbrace{H_{0}(a,H_{0}(a,\ldots H_{0}(a,H_{0}(a,H_{1}(a,b)))\ldots))}_{c\ \text{copies of}\ H_{0}}\\
&=H_{1}(H_{1}(a,b),c).
\end{align}
\end{subequations}
\item\textit{Closure of multiplication.}
\begin{align}
H_{2}(a,b)=\underbrace{H_{1}(a,H_{1}(\ldots H_{1}(a,H_{1}(a,a))\ldots))}_{b\ \text{copies of}\ a}\in\mathbb{N}.
\end{align}
\item\textit{Multiplicative identity.}
\begin{subequations}
\begin{align}
H_{2}(a,1)&=H_{1}(a,H_{2}(a,0))\\
&=H_{1}(a,0)\\
&=a
\end{align}
\end{subequations}
and
\begin{subequations}
\begin{align}
H_{2}(1,a)&=\underbrace{H_{1}(1,H_{1}(\ldots H_{1}(1,H_{1}(1,1))\ldots))}_{a-1\ \text{copies of}\ H_{1}}\\
&=\underbrace{H_{1}(1,H_{1}(\ldots H_{1}(1,H_{1}(1,2))\ldots))}_{a-2\ \text{copies of}\ H_{1}}\\
&=\underbrace{H_{1}(1,H_{1}(\ldots H_{1}(1,H_{1}(1,3))\ldots))}_{a-3\ \text{copies of}\ H_{1}}\\
&\vdots\nonumber\\
&=H_{1}(1,a-1)\\
&=a.
\end{align}
\end{subequations}
\item \textit{Distributivity.}  Let $H_{1}(b,c)>0$, then
\begin{subequations}
\begin{align}
H_{2}(a,H_{1}(b,c))&=\underbrace{H_{1}(a,H_{1}(a,\ldots H_{1}(a,H_{1}(a,a))\ldots))}_{H_{1}(b,c)-1\ \text{copies of}\ H_{1}}\\
&=\underbrace{H_{1}(a,H_{1}(a,\ldots H_{1}(a,H_{2}(a,2))\ldots))}_{H_{1}(b,c)-2\ \text{copies of}\ H_{1}}\\
&=\underbrace{H_{1}(a,H_{1}(a,\ldots H_{1}(a,H_{2}(a,3))\ldots))}_{H_{1}(b,c)-3\ \text{copies of}\ H_{1}}\\
&\vdots\nonumber\\
&=\underbrace{H_{1}(a,H_{1}(a,\ldots H_{1}(a,H_{1}(a,H_{2}(a,c)))\ldots))}_{b\ \text{copies of}\ H_{1}}\label{semif:aso:in}\\
&=\underbrace{H_{1}(a,H_{1}(a,\ldots H_{1}(a,H_{1}(H_{1}(a,a),H_{2}(a,c)))\ldots))}_{b\ \text{copies of}\ H_{1}}\label{semif:aso:out}\\
&=\underbrace{H_{1}(a,H_{1}(a,\ldots H_{1}(a,H_{1}(H_{2}(a,2),H_{2}(a,c)))\ldots))}_{b-1\ \text{copies of}\ H_{1}}\\
&=\underbrace{H_{1}(a,\ldots H_{1}(a,H_{1}(H_{1}(a,H_{2}(a,2)),H_{2}(a,c)))\ldots))}_{b-1\ \text{copies of}\ H_{1}}\\
&=\underbrace{H_{1}(a,\ldots H_{1}(a,H_{1}(H_{1}(a,H_{2}(a,3)),H_{2}(a,c)))\ldots))}_{b-2\ \text{copies of}\ H_{1}}\\
&\vdots\nonumber\\
&=H_{1}(H_{1}(a,H_{2}(a,b-1)),H_{2}(a,c))\\
&=H_{1}(H_{2}(a,b),H_{2}(a,c)).
\end{align}
\end{subequations}
To prove $H_{2}(H_{1}(b,c),a)=H_{1}(H_{2}(b,a),H_{2}(c,a))$, one can proceed by induction. Note that
\begin{subequations}
\begin{align}
H_{2}(H_{1}(b,c),1)&=H_{2}(1,H_{1}(b,c))\\
&=H_{1}(H_{2}(1,b),H_{2}(1,c))\\
&=H_{1}(H_{2}(b,1),H_{2}(c,1)).
\end{align}
\end{subequations}
Assume $H_{2}(H_{1}(b,c),a)=H_{1}(H_{2}(b,a),H_{2}(c,a))$, then
\begin{subequations}
\begin{align}
H_{2}(H_{1}(b,c),H_{1}(a,1))&=H_{1}(H_{2}(H_{1}(b,c),a),H_{2}(H_{1}(b,c),1))\\
&=H_{1}(H_{2}(H_{1}(b,c),a),H_{1}(b,c))\\
&=H_{1}(H_{1}(H_{2}(b,a),H_{2}(c,a)),H_{1}(b,c))\\
&=H_{1}(H_{1}(H_{2}(b,a),b),H_{1}(H_{2}(c,a),c))\\
&=H_{1}(H_{1}(H_{2}(b,a),H_{2}(b,1)),H_{1}(H_{2}(c,a),H_{2}(c,1)))\\
&=H_{1}(H_{2}(b,H_{1}(a,1)),H_{2}(c,H_{1}(a,1))).
\end{align}
\end{subequations}
The case $H_{1}(b,c)=0$ is straightforward since it implies $b=c=0$.
\item\textit{Commutativity of multiplication.} By induction. It is known that $H_{2}(a,1)=H_{2}(1,a)$.  Assume $H_{2}(a,b)=H_{2}(b,a)$, then
\begin{subequations}
\begin{align}
H_{2}(a,H_{1}(1,b))&=H_{1}(H_{2}(a,1),H_{2}(a,b))\\
&=H_{1}(H_{2}(1,a),H_{2}(b,a))\\
&=H_{2}(H_{1}(1,b),a).
\end{align}
\end{subequations}
\item\textit{Associativty of multiplication.}  By induction. Note that
\begin{subequations}
\begin{align}
H_{2}(H_{2}(a,b),0)&=0\\
&=H_{2}(a,0)\\
&=H_{2}(a,H_{2}(b,0)).
\end{align}
\end{subequations}
Assume $H_{2}(H_{2}(a,b),c)=H_{2}(a,H_{2}(b,c))$, then
\begin{subequations}
\begin{align}
H_{2}(H_{2}(a,b),H_{1}(1,c))&=H_{1}(H_{2}(H_{2}(a,b),1),H_{2}(H_{2}(a,b),c))\\
&=H_{1}(H_{2}(a,b),H_{2}(H_{2}(a,b),c))\\
&=H_{1}(H_{2}(a,b),H_{2}(a,H_{2}(b,c)))\\
&=H_{2}(a,H_{1}(b,H_{2}(b,c)))\\
&=H_{2}(a,H_{2}(b,H_{1}(1,c))).
\end{align}
\end{subequations}
\end{enumerate}
\end{proof} 
\begin{definition}\label{defn:exp-int}
Define the exponential function $E:\mathbb{N}\rightarrow\mathbb{N}:E(a)=H_{3}(w_{0},a)$ for a given base $w_{0}\in\mathbb{N}\setminus\{0,1\}$.  Denote by $L$ its inverse function.
\end{definition}
\begin{notation}
For $n\in\mathbb{N}$, denote $\mathord{\stackrel{n+1}{\mathbb{N}}}=E(\mathord{\stackrel{n}{\mathbb{N}}})$ with $\mathord{\stackrel{0}{\mathbb{N}}}=\mathbb{N}$.  Denote $\mathord{\stackrel{n}{\mathbb{N}}}=\{\mathord{\stackrel{n}{0}},\mathord{\stackrel{n}{1}},\mathord{\stackrel{n}{2}},\ldots\}$ such that for $\mathord{\stackrel{n+1}{k}}\in\mathord{\stackrel{n+1}{\mathbb{N}}}$ and $\mathord{\stackrel{n}{k}}\in\mathord{\stackrel{n}{\mathbb{N}}}$, one obtains $\mathord{\stackrel{n+1}{k}}=E(\mathord{\stackrel{n}{k}})$. 
\end{notation}
The following results are only valid for those values of $\mathord{\stackrel{n}{\mathbb{N}}}$ such that $n\in\mathbb{N}$.
\begin{remark}
The function $E$ generates a descending filtration in $\mathbb{N}$, meaning $\mathord{\stackrel{n+1}{\mathbb{N}}}\subset\mathord{\stackrel{n}{\mathbb{N}}}$.
\end{remark}
\begin{proposition}\label{prop:mrule}
Let $a,b\in\mathord{\stackrel{1}{\mathbb{N}}}$, then
\begin{equation}
H_{2}(a,b)=E(H_{1}(L(a),L(b))).
\end{equation}
\end{proposition}
\begin{proof}
Let $a,b\in\mathord{\stackrel{1}{\mathbb{N}}}$, then
\begin{subequations}
\begin{align}
E(H_{1}(L(a),L(b)))&=H_{3}(w_{0},H_{1}(L(a),L(b)))\\
&=\underbrace{H_{2}(w_{0},H_{2}(w_{0},\ldots H_{2}(w_{0},H_{2}(w_{0},w_{0}))\ldots))}_{H_{1}(L(a),L(b))-1\ \text{copies of}\ H_{2}}\\
&=\underbrace{H_{2}(w_{0},H_{2}(w_{0},\ldots H_{2}(w_{0},H_{3}(w_{0},2))\ldots))}_{H_{1}(L(a),L(b))-2\ \text{copies of}\ H_{2}}\\
&=\underbrace{H_{2}(w_{0},H_{2}(w_{0},\ldots H_{2}(w_{0},H_{3}(w_{0},3))\ldots))}_{H_{1}(L(a),L(b))-3\ \text{copies of}\ H_{2}}\\
&\vdots\nonumber\\
&=\underbrace{H_{2}(w_{0},H_{2}(w_{0},\ldots H_{2}(w_{0},H_{3}(w_{0},L(a)))\ldots))}_{L(b)\ \text{copies of}\ H_{2}}\\
&=\underbrace{H_{2}(w_{0},H_{2}(w_{0},\ldots H_{2}(w_{0},H_{2}(a,w_{0}))\ldots))}_{L(b)\ \text{copies of}\ H_{2}}\\
&=\underbrace{H_{2}(w_{0},H_{2}(w_{0},\ldots H_{2}(w_{0},H_{2}(a,H_{3}(w_{0},2)))\ldots))}_{L(b)-1\ \text{copies of}\ H_{2}}\\
&=\underbrace{H_{2}(w_{0},H_{2}(w_{0},\ldots H_{2}(w_{0},H_{2}(a,H_{3}(w_{0},3)))\ldots))}_{L(b)-2\ \text{copies of}\ H_{2}}\\
&\vdots\nonumber\\
&=H_{2}(w_{0},H_{2}(a,H_{3}(w_{0},L(b)-1)))\\
&=H_{2}(a,H_{3}(w_{0},L(b)))\\
&=H_{2}(a,b).
\end{align}
\end{subequations}
\end{proof}
\begin{definition}\label{defn:ahyp}
 The sequence of commutative hyperoperations generated from the abelian monoid $(\mathbb{N},H_{2})$ is the sequence of binary operations $F_{n}:(\mathord{\stackrel{n-1}{\mathbb{N}}})^{2}\rightarrow\mathord{\stackrel{n-1}{\mathbb{N}}}$ for $n=2,3,4,\dotsc$, defined recursively as
 \begin{equation}
 F_{n}(a,b)=\left\{\begin{array}{ll}
E(H_{2}(L(a),L(b))),&\text{if}\ n=2;\\
 E(F_{n-1}(L(a),L(b))),&\text{otherwise}.
 \end{array}\right.
 \end{equation}
For $a,b\in\mathord{\stackrel{0}{\mathbb{N}}}$, denote
 \begin{equation}\label{eq:f1f2}
 F_{n}(a,b)=\left\{\begin{array}{ll}
H_{1}(a,b),&\text{if}\ n=0;\\
H_{2}(a,b),&\text{if}\ n=1.
 \end{array}\right.
 \end{equation}
 \end{definition}
In Bennett's original report \cite{BE15} is less constructive there is no discussion of the sets $\mathord{\stackrel{n}{\mathbb{N}}}$, instead focusing on the usual real and complex numbers. 
\begin{theorem}\label{th:Fsemiring}
The tuple $(\mathord{\stackrel{n}{\mathbb{N}}},F_{n},F_{n+1})$ is a commutative semiring.
\end{theorem}
\begin{proof}
By induction.  From Proposition~\ref{prop:Hsemiring} the tuple $(\mathbb{N},H_{1},H_{2})$ is a commutative semiring and by definition $(\mathbb{N},H_{1},H_{2})=(\mathord{\stackrel{0}{\mathbb{N}}},F_{0},F_{1})$.  Assume the tuple $(\mathord{\stackrel{n-1}{\mathbb{N}}},F_{n-1},F_{n})$ is a commutative semiring, with additive identity $\mathord{\stackrel{n-1}{0}}$ and multiplicative identity $\mathord{\stackrel{n-1}{1}}$.  The following is the proof that $(\mathord{\stackrel{n}{\mathbb{N}}},F_{n},F_{n+1})$ is a commutative semiring. Let $a,b,c\in\mathord{\stackrel{n}{\mathbb{N}}}$.
\begin{enumerate}
\item \textit{Closure of addition.}
\begin{equation}
F_{n}(a,b)=E(F_{n-1}(L(a),L(b)))\in E(\mathord{\stackrel{n-1}{\mathbb{N}}})=\mathord{\stackrel{n}{\mathbb{N}}}.
\end{equation}
\item \textit{Additive identity.} 
\begin{subequations}
\begin{align}
F_{n}(a,\mathord{\stackrel{n}{0}})
&=E(F_{n-1}(L(a),\mathord{\stackrel{n-1}{0}}))\\
&=E(L(a))\\
&=a.
\end{align}
\end{subequations}
\item \textit{Commutativity of addition.}
\begin{subequations}
\begin{align}
F_{n}(a,b)&=E(F_{n-1}(L(a),L(b)))\\
&=E(F_{n-1}(L(b),L(a)))\\
&=F_{n}(b,a).
\end{align}
\end{subequations}
\item \textit{Associativity of addition.}
\begin{subequations}
\begin{align}
F_{n}(a,F_{n}(b,c))&=E(F_{n-1}(L(a),L(F_{n}(b,c))))\\
&=E(F_{n-1}(L(a),L(E(F_{n-1}(L(b),L(c))))))\\
&=E(F_{n-1}(L(a),F_{n-1}(L(b),L(c))))\\
&=E(F_{n-1}(F_{n-1}(L(a),L(b)),L(c)))\\
&=F_{n}(E(F_{n-1}(L(a),L(b))),E(L(c)))\\
&=F_{n}(F_{n}(E(L(a)),E(L(b))),c)\\
&=F_{n}(F_{n}(a,b),c).
\end{align}
\end{subequations}
%
\item \textit{Closure of multiplication.}
\begin{equation}
F_{n+1}(a,b)=E(F_{n}(L(a),L(b)))\in E(\mathord{\stackrel{n-1}{\mathbb{N}}})=\mathord{\stackrel{n}{\mathbb{N}}}.
\end{equation}
\item \textit{Multiplicative identity.}  
\begin{subequations}
\begin{align}
F_{n+1}(a,\mathord{\stackrel{n}{1}})
&=E(F_{n}(L(a),\mathord{\stackrel{n-1}{1}}))\\
&=E(L(a))\\
&=a.
\end{align}
\end{subequations}
\item \textit{Distributivity.}
\begin{subequations}
\begin{align}
F_{n+1}(a,F_{n}(b,c))
&=E(F_{n}(L(a),L(F_{n}(b,c))))\\
&=E(F_{n}(L(a),F_{n-1}(L(b),L(c))))\\
&=E(F_{n-1}(F_{n}(L(a),L(b)),F_{n}(L(a),L(c))))\\
&=F_{n}(F_{n+1}(a,b),F_{n+1}(a,c)).
\end{align}
\end{subequations}
\item \textit{Commutativity of multiplication.}
\begin{subequations}
\begin{align}
F_{n+1}(a,b)&=E(F_{n}(L(a),L(b)))\\
&=E(F_{n}(L(b),L(a)))\\
&=F_{n+1}(b,a).
\end{align}
\end{subequations}
%
\item \textit{Associativity of multiplication.}
\begin{subequations}
\begin{align}
F_{n+1}(a,F_{n+1}(b,c))&=E(F_{n}(L(a),L(F_{n+1}(b,c))))\\
&=E(F_{n}(L(a),L(E(F_{n}(L(b),L(c))))))\\
&=E(F_{n}(L(a),F_{n}(L(b),L(c))))\\
&=E(F_{n}(F_{n}(L(a),L(b)),L(c)))\\
&=F_{n+1}(E(F_{n}(L(a),L(b))),E(L(c)))\\
&=F_{n+1}(F_{n+1}(E(L(a)),E(L(b))),c)\\
&=F_{n+1}(F_{n+1}(a,b),c).
\end{align}
\end{subequations}
\end{enumerate}
\end{proof}
\begin{remark}
The function $E$ is the generator of isomorphisms in the sequence of semirings $\{(\mathord{\stackrel{n}{\mathbb{N}}},F_{n},F_{n+1})\}_{n\in\mathbb{N}}$.  
\end{remark}
\begin{remark} 
The monoid $(\mathord{\stackrel{n}{\mathbb{N}}},F_{n})$ is a submonoid of $(\mathord{\stackrel{n-1}{\mathbb{N}}},F_{n})$.
\end{remark}
\begin{definition}
Define the ordering relation $\stackrel{n}{<}$ in the monoid $(\mathord{\stackrel{n}{\mathbb{N}}},F_{n})$ as
\begin{equation}
a\stackrel{n}{<}b\iff\exists c\in\mathord{\stackrel{n}{\mathbb{N}}}\setminus\{\mathord{\stackrel{n}{0}}\}:F_{n}(a,c)=b.
\end{equation}
\end{definition}
\begin{proposition}\label{prop:mono}
Let $a,b\in\mathord{\stackrel{n}{\mathbb{N}}}$, then
\begin{equation}
a\stackrel{n}{<}b\iff E(a)\stackrel{n}{<}E(b).
\end{equation}
\end{proposition}
\begin{proof}
The proof that $a\stackrel{n}{<}b\Rightarrow E(a)\stackrel{n}{<}E(b)$ is the following.  The relation $a\stackrel{n}{<}b$ is equivalent to the existence of an element $c\in\mathord{\stackrel{n}{\mathbb{N}}}$ such that $F_{n}(a,c)=b$.  The latter is equivalent to the equation $F_{n+1}(E(a),E(c))=E(b)$.  Since $F_{n+1}(E(a),E(c))=F_{n}(E(a),c')$, where $c'=S^{-1}(E(c))$,
then $E(a)\stackrel{n}{<}E(b)$.  The converse statement is proven in a similar way.
\end{proof}
\begin{proposition}\label{prop:monday}
Let $a,b\in\mathord{\stackrel{n}{\mathbb{N}}}$, then
\begin{equation}
a\stackrel{n-1}{<}b\iff a\stackrel{n}{<}b.
\end{equation}
\end{proposition}
\begin{proof}
Let $a,b\in\mathord{\stackrel{n}{\mathbb{N}}}\subset\mathord{\stackrel{n-1}{\mathbb{N}}}$. The proof that $a\stackrel{n-1}{<}b\Rightarrow a\stackrel{n}{<}b$ is the following: If $a\stackrel{n-1}{<}b$, the strict monotonicity of $L$ in $\mathord{\stackrel{n-1}{\mathbb{N}}}$, from Proposition \ref{prop:mono}, implies $L(a)\stackrel{n-1}{<}L(b)$, thereby $\exists c':F_{n-1}(L(a),c')=L(b)$.  Without loss of generality let $c'=L(c)$, where $c\in\mathord{\stackrel{n}{\mathbb{N}}}$.  Then
\begin{equation}
E(F_{n-1}(L(a),L(c)))=E(L(b)),
\end{equation}
i.e., $F_{n}(a,c)=b$, thereby $a\stackrel{n}{<}b$.  The proof that $a\stackrel{n-1}{<}b\Leftarrow a\stackrel{n}{<}b$ is the following: If $a\stackrel{n}{<}b$, then $\exists c\in\mathord{\stackrel{n}{\mathbb{N}}}$ such that $F_{n}(a,c)=b$, that is
\begin{equation}
E(F_{n-1}(L(a),L(c)))=E(L(b)).
\end{equation}
Since $E$ is an injective function, then
\begin{equation}
F_{n-1}(L(a),L(c))=L(b),
\end{equation}
implying $L(a)\stackrel{n-1}{<}L(b)$.  Since $L$ is strictly monotone in $\mathord{\stackrel{n-1}{\mathbb{N}}}$, then $a\stackrel{n-1}{<}b$.
\end{proof}
\begin{remark}
The ordering relation $\stackrel{n}{<}$ in the monoid $(\mathord{\stackrel{n}{\mathbb{N}}},F_{n})$ is equivalent to the restriction of the standard ordering $<$ in $\mathord{\stackrel{n}{\mathbb{N}}}$ as a subset of $\mathbb{N}$.
\end{remark}
\begin{remark}
Since $\mathord{\stackrel{n}{1}}=\mathord{\stackrel{n+1}{0}}$, all identity elements can be written as $\mathord{\stackrel{n}{0}}$, for a given $n\in\mathbb{N}$.  They can be generated as
\begin{subequations}\label{eq:tetration}
\begin{align}
\mathord{\stackrel{0}{0}}&=H_{4}(w_{0},-1)=0,\\
\mathord{\stackrel{1}{0}}&=H_{4}(w_{0},0)=1,\\
\mathord{\stackrel{2}{0}}&=H_{4}(w_{0},1)=w_{0},\\
&\vdots\nonumber\\
\mathord{\stackrel{n+1}{0}}&=H_{4}(w_{0},n)
=\underbrace{H_{3}(w_{0},H_{3}(w_{0},\dotsc, H_{3}(w_{0},w_{0})))}_{n\ \text{copies of}\ w_{0}},\\
&\vdots\nonumber
\end{align}
\end{subequations}
In the expression $H_{4}(w_{0},n)$, the integer argument $w_{0}$ is known as the \textit{base} and the integer argument $n$ is known as the \textit{height} or \textit{level}.  For completness, in the Eq.~\eqref{eq:tetration} the definition of $H_{4}$ has been extended to allow for negative integers in the height.  The operation $H_{4}$ is known as \textit{tetration}, restricted to the natural numbers.
\end{remark}
Underlying the sequence of commutative hyperoperations is the fact that any injective function $f$ with domain in a monoid, generates an isomorphism between monoids.  Let $a,b\in f(\mathord{\stackrel{0}{\mathbb{N}}})$.  The right conjugation of ordinary addition $F_{0}$ by $f$ is
\begin{equation}
{\rm Ad}_{f}(F_{0})(a,b)=f(F_{0}(f^{-1}(a),f^{-1}(b)).
\end{equation}
The operation ${\rm Ad}_{f}(F_{0})$ forms a monoid in $f(\mathord{\stackrel{0}{\mathbb{N}}})$ which is isomorphic to $(\mathord{\stackrel{0}{\mathbb{N}}},F_{0})$.  In fact all properties studied so far can be reproduced in this new setting; if one starts from the successor function ${\rm Ad}_{f}(S)$ in $f(\mathbb{N})$, it is easy to prove by induction that the corresponding regular hyperoperations take the form ${\rm Ad}_{f}(H_{n})$.
Further assuming $f$ is composable with itself, the tuples $\{(f^{n}(\mathord{\stackrel{0}{\mathbb{N}}}),{\rm Ad}_{f^{n}}(F_{0}))\}_{n\in\mathbb{N}}$ become a sequence of commutative monoids.  A sequence of commutative semirings can be formed as $\{(f^{n}(\mathord{\stackrel{0}{\mathbb{N}}}),{\rm Ad}_{f^{n}}(F_{0}),{\rm Ad}_{f^{n}}(F_{1}))\}_{n\in\mathbb{N}}$, with $f$ as the generator of semiring isomorphisms. A special case is when the function $f$ is \textit{exponential} along the sequence of semirings, meaning that $f^{n}(\mathord{\stackrel{0}{\mathbb{N}}})\cap f^{n+1}(\mathord{\stackrel{0}{\mathbb{N}}})$ is a nonempty set where,
\begin{equation}
{\rm Ad}_{f^{n}}(F_{1})={\rm Ad}_{f^{n+1}}(F_{0}).
\end{equation}
In Fig.~\ref{fig:1} is presented a visualization that contrasts the case of exponential vs.~nonexponential $f$.
\begin{figure}[h]
     \centering
     \begin{subfigure}[b]{0.49\textwidth}
         \centering
         \includegraphics[trim=35 0 35 0,clip,width=0.85\textwidth]{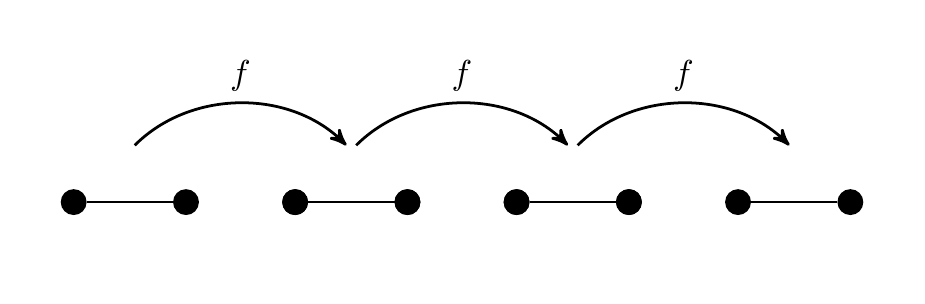}
         \caption{Nonexponential $f$.}
         \label{fig:1a}
     \end{subfigure}
     \hfill
     \begin{subfigure}[b]{0.49\textwidth}
         \centering
         \includegraphics[trim=70 0 70 0,clip,width=0.45\textwidth]{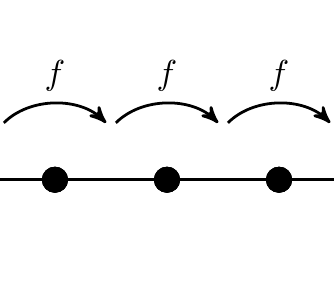}
         \caption{Exponential $f$.}
         \label{fig:1b}
     \end{subfigure}
        \caption{Visualizing a sequence of semirings: A semiring (two disks connected by a line) is comprised of two operations: addition (left-hand disk) and multiplication (right-hand disk). (a) A nonexponential injective function $f$ induces a sequence where nearby semirings do not share operations. (b)  An exponential function $f$ induces a sequence where nearby semirings share one operation.}
        \label{fig:1}
\end{figure}

%% file: 2-new-seq.tex
This section presents new sequences of hyperoperations originated from Peano systems initially associated to the semirings $\{(\mathord{\stackrel{n}{\mathbb{N}}},F_{n},F_{n+1})\}_{n\in\mathbb{N}}$.  The commutative hyperoperations $\{F_{n}\}_{n\in\mathbb{N}}$ are shown to be regular hyperoperations.  A recursive procedure to create new sequences of hyperoperations is stablished.
\begin{definition}\label{def:seneca}
Define the successor function in $\mathord{\stackrel{n}{\mathbb{N}}}$ as the recursive function
\begin{equation}
S_{n}:\mathord{\stackrel{n}{\mathbb{N}}}\rightarrow\mathord{\stackrel{n}{\mathbb{N}}}\setminus\{\mathord{\stackrel{n}{0}}\}:S_{n}(a)=\left\{\begin{array}{ll}S(a),&\text{if}\ n=0;\\E(S_{n-1}(L(a))),&\text{if}\ n>0.\end{array}\right.
\end{equation}
Denote its inverse function as $S_{n}^{-1}$.
\end{definition}
\begin{remark}
The tuples $\{(\mathord{\stackrel{n}{\mathbb{N}}},\mathord{\stackrel{n}{0}},S_{n})\}_{n\in\mathbb{N}}$ is a sequence of Peano systems with initial element $(\mathbb{N},0,S)$.
\end{remark}
\begin{definition}\label{def:liverpool}
 The sequence of regular hyperoperations generated from the succesor in $\mathord{\stackrel{n}{\mathbb{N}}}$ is the sequence of binary operations $[n:m]:(\mathord{\stackrel{n}{\mathbb{N}}})^{2}\rightarrow\mathord{\stackrel{n}{\mathbb{N}}}$, for $m\in\mathbb{N}$, defined recursively as
 \begin{equation}
 a\:[n:m]\:b=\left\{\begin{array}{ll}
 S_{n}(b),&\text{if}\ m=0;\\ 
 a,&\text{if}\ m=1\ \text{and}\ b=\mathord{\stackrel{n}{0}};\\
 \mathord{\stackrel{n}{0}},&\text{if}\ m=2\ \text{and}\ b=\mathord{\stackrel{n}{0}};\\
 \mathord{\stackrel{n}{1}},&\text{if}\ m\geq 3\ \text{and}\ b=\mathord{\stackrel{n}{0}};\\
 a\:[n:m-1]\:(a\:[n:m]\:S_{n}^{-1}(b)),&\text{otherwise}.
 \end{array}\right.
 \end{equation}
Note that $a\:[0:m]\:b=H_{m}(a,b)$ for $a,b\in\mathord{\stackrel{0}{\mathbb{N}}}$.
\end{definition}
The following results are valid for those values of $[n:m]$ such that $n,m\in\mathbb{N}$.
\begin{proposition}
The tuple $(\mathord{\stackrel{n}{\mathbb{N}}},[n:1],[n:2])$ is a commutative semiring with additive identity $\mathord{\stackrel{n}{0}}$ and multiplicative identity $\mathord{\stackrel{n}{1}}$.
\end{proposition}
\begin{proof}
The case $n=0$ is proven in Theorem~\ref{prop:Hsemiring}, in account that $a\:[0:i]\:b=H_{i}(a,b)$, for $i=1,2$. The remaining cases are proven analogously, using the Peano system $(\mathord{\stackrel{n}{\mathbb{N}}},\mathord{\stackrel{n}{0}},S_{n})$.
\end{proof}
\begin{theorem}\label{th:oxford}
Let $a,b\in\mathord{\stackrel{n}{\mathbb{N}}}$, then
\begin{equation}
a\:[n:m]\:b=E(L(a)\:[n-1:m]\:L(b)).
\end{equation}
\end{theorem}
\begin{proof}
By induction.  Let $a,b\in\mathord{\stackrel{n}{\mathbb{N}}}$.  Without loss of generality let $b=\mathord{\stackrel{n}{k}}$ where $k\in\mathbb{N}$.  The case $m=0$ is trivial.  For $m=1$, one obtains
\begin{subequations}
\begin{align}
a\:[n:1]\:b&=\underbrace{a\:[n:0]\:(a\:[n:0]\:\dotsb (a\:[n:0]\:(a\:[n:0]\:a))\dotsb)}_{k\ \text{copies of}\ [n:0]}\\
&=S_{n}^{k}(a)\\
&=E\circ S_{n-1}^{k}(L(a))\\
&=E(L(a)\:[n-1:1]\:L(b)).
\end{align}
\end{subequations}
Assume
\begin{equation}
a\:[n:m-1]\:b=E(L(a)\:[n-1:m-1]\:L(b)),
\end{equation}
then for $m>1$, one obtains 
\begin{subequations}
\begin{gather}
a\:[n:m]\:b=\underbrace{a\:[n:m-1]\:(a\:[n:m-1]\:\dotsb (a\:[n:m-1]\:(a\:[n:m-1]\:a))\dotsb)}_{L^{n}(b)\ \text{copies of}\ a}\\
=E(\underbrace{L(a)\:[n-1:m-1]\:(L(a)\:[n-1:m-1]\:\dotsb (L(a)\:[n-1:m-1]\:(L(a)\:[n-1:m-1]\:L(a)))\dotsb)}_{L^{n}(b)\ \text{copies of}\ L(a)})\\
=E(L(a)\:[n-1:m]\:L(b)).
\end{gather}
\end{subequations}
\end{proof}
The following corollary states that the commutative hyperoperations from Definition~\ref{defn:ahyp} are regular hyperoperations for a given Peano system, as depicted in Definition~\ref{def:liverpool}.
\begin{corollary}\label{coro:cromwell}
Let $a,b\in\mathord{\stackrel{n}{\mathbb{N}}}$, then
\begin{equation}
a\:[n:2]\:b=F_{n+1}(a,b).
\end{equation}
\end{corollary}
\begin{proof}
By induction.  Let $a,b\in\mathord{\stackrel{0}{\mathbb{N}}}$.  From Definitions~\ref{defn:ahyp} and \ref{def:liverpool} one obtains that $F_{1}(a,b)=H_{2}(a,b)$ and $a\:[0:2]\:b=H_{2}(a,b)$, therefore $a\:[0:2]\:b=F_{1}(a,b)$.  For  $a,b\in\mathord{\stackrel{n-1}{\mathbb{N}}}$, assume
\begin{equation}
a\:[n-1:2]\:b=F_{n}(a,b).
\end{equation}
Let $a,b\in\mathord{\stackrel{n}{\mathbb{N}}}$.  From Theorem~\ref{th:oxford} one obtains
\begin{subequations}
\begin{align}
a\:[n:2]\:b&=E(L(a)\:[n-1:2]\:L(b))\\
&=E(F_{n}(L(a),L(b)))\\
&=F_{n+1}(a,b).
\end{align}
\end{subequations}
\end{proof} 
\begin{lemma}\label{lem:cornell} 
Let $a,b\in\mathord{\stackrel{n}{\mathbb{N}}}$, then
\begin{equation}
a\:[n:1]\:b=\underbrace{a\:[n-1:1]\:\dotsb\:[n-1:1]\:a}_{L^{n-1}(b)\ \text{copies of}\ a}.
\end{equation}
\end{lemma}
\begin{proof}
By induction.  Let $a,b\in\mathord{\stackrel{1}{\mathbb{N}}}$.  From Theorem~\ref{th:oxford} and Proposition~\ref{prop:mrule}, one obtains
\begin{subequations}
\begin{align}
a\:[1:1]\:b&=E(L(a)\:[0:1]\:L(b))\\
&=E(H_{1}(L(a),L(b)))\\
&=H_{2}(a,b)\\
&=a\:[0:2]\:b\\
&=\underbrace{a\:[0:1]\:\dotsb\:[0:1]\:a}_{b\ \text{copies of}\ a}.
\end{align}
\end{subequations}
 For $a,b\in\mathord{\stackrel{n-1}{\mathbb{N}}}$, assume
\begin{equation}
a\:[n-1:1]\:b=\underbrace{a\:[n-2:1]\:\dotsb\:[n-2:1]\:a}_{L^{n-2}(b)\ \text{copies of}\ a}.
\end{equation}
For $a,b\in\mathord{\stackrel{n}{\mathbb{N}}}$, then
\begin{subequations}
\begin{align}
a\:[n:1]\:b&=E(L(a)\:[n-1:1]\:L(b))\\
&=E(\underbrace{L(a)\:[n-2:1]\:\dotsb\:[n-2:1]\:L(a)}_{L^{n-1}(b)\ \text{copies of}\ L(a)})\\
&=\underbrace{a\:[n-1:1]\:\dotsb\:[n-1:1]\:a}_{L^{n-1}(b)\ \text{copies of}\ a}.
\end{align}
\end{subequations}
\end{proof}
\begin{proposition}\label{prop:janeiro} 
Let $a,b\in\mathord{\stackrel{n+1}{\mathbb{N}}}$, then 
\begin{equation}
a\:[n:2]\:b=a\:[n+1:1]\:b.
\end{equation}
\end{proposition}
\begin{proof}
Let $a,b\in\mathord{\stackrel{n+1}{\mathbb{N}}}$.  From Lemma~\ref{lem:cornell} one obtains
\begin{subequations}
\begin{align}
a\:[n:2]\:b
&=\underbrace{a\:[n:1]\:\dotsb\:[n:1]\:a}_{L^{n}(b)\ \text{copies of}\ a}\\
&=a\:[n+1:1]\:b.
\end{align}
\end{subequations}
\end{proof}
\begin{proposition}\label{prop:3to2} 
Let $a\in\mathord{\stackrel{n+1}{\mathbb{N}}}$ and $b\in\mathord{\stackrel{n}{\mathbb{N}}}$, then
\begin{equation}
a\:[n:3]\:b=a\:[n+1:2]\:E(b).
\end{equation}
\end{proposition} 
\begin{proof}
From Proposition~\ref{prop:janeiro}, one obtains
\begin{subequations}
\begin{align}
a\:[n:3]\:b&=\underbrace{a\:[n:2]\:\dotsb\:[n:2]\:a}_{L^{n}(b)\ \text{copies of}\ a}\\
&=\underbrace{a\:[n+1:1]\:\dotsb\:[n+1:1]\:a}_{L^{n}(b)\ \text{copies of}\ a}\\
&=a\:[n+1:2]\:E(b).
\end{align}
\end{subequations}
\end{proof}
\begin{proposition}\label{prop:manchester}
Let $a,b\in\mathord{\stackrel{n}{\mathbb{N}}}$, then
\begin{equation}\label{eq:manchester}
a\:[n:1]\:b=F_{n}(a,b).
\end{equation}
\end{proposition}
\begin{proof}
Let $a,b\in\mathord{\stackrel{0}{\mathbb{N}}}$.  From Definitions~\ref{defn:ahyp} and \ref{def:liverpool} one obtains that $F_{0}(a,b)=H_{1}(a,b)$ and $a\:[0:1]\:b=H_{1}(a,b)$, therefore $a\:[0:1]\:b=F_{0}(a,b)$.  For $a,b\in\mathord{\stackrel{n}{\mathbb{N}}}$ with $n\in\mathbb{N}\setminus\{0\}$, the Eq.~\eqref{eq:manchester} follows from Corollary~\ref{coro:cromwell} and Proposition~\ref{prop:janeiro}.
\end{proof}
The sequence $\{(\mathord{\stackrel{n}{\mathbb{N}}},\mathord{\stackrel{n}{0}},S_{n})\}_{n\in\mathbb{N}}$ is generated from the system $(\mathord{\stackrel{0}{\mathbb{N}}},\mathord{\stackrel{0}{0}},S_{0})$ via the exponential function $E$.  In general, a Peano system induces an exponential semiring from its regular hyperoperations.  The exponential function of this semiring is used to generate a sequence of Peano systems.  This is a recursive procedure.  Identifying a new sequence requires recording previous choices of generating systems.
\begin{notation}
The index of a sequence of Peano systems of level $k$, is an element of the set
\begin{equation}
I_{{\bf n}_{k-1}}=\left\{\begin{array}{ll}
\mathbb{N},&\text{if}\ k=1;\\
\{({\bf n}_{k-1},0),({\bf n}_{k-1},1),({\bf n}_{k-1},2),\dotsc\},&\text{otherwise};
\end{array}\right.
\end{equation}
where ${\bf n}_{k-1}=(n_{1},n_{2},\dotsc,n_{k-1})$ is a fixed element in $\mathbb{N}^{k-1}$, recording previous choices of Peano systems.  For $k>1$, let ${\bf n}_{k}=({\bf n}_{k-1},n_{k})\in I_{{\bf n}_{k-1}}$ and $l\in\mathbb{Z}$.  Whenever $n_{k}+l\in\mathbb{N}$, denote
\begin{equation}
{\bf n}_{k}+l=({\bf n}_{k-1},n_{k}+l).
\end{equation}
\end{notation}
The present proposal to generate new sequences of Peano systems can be visualized as a binary tree diagram, see Fig.~\ref{fig:diag}.
\begin{figure}[h]
\centering
\includegraphics[width=0.5\textwidth]{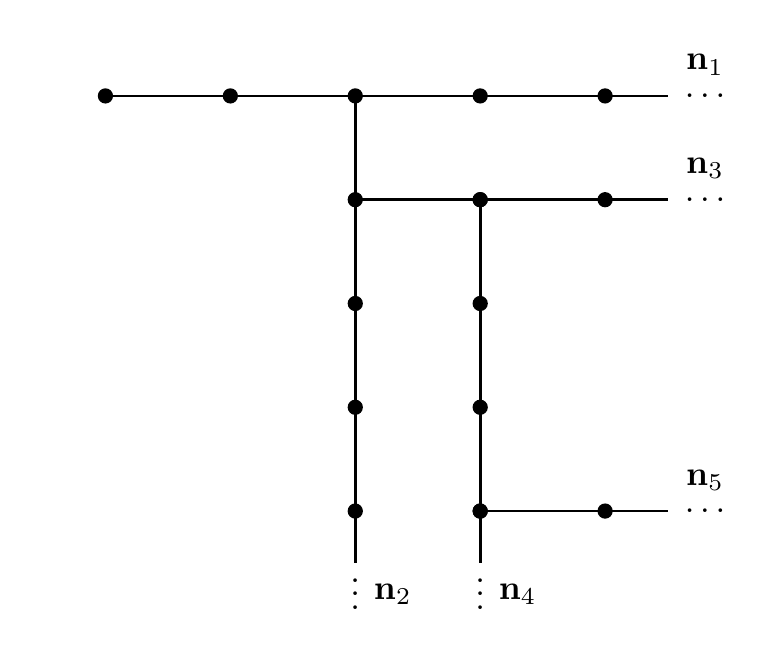}
\caption{The diagram shows five sequences of Peano systems.  In particular, the sequence with index ${\bf n}_{5}$ is obtained after the choices: $n_{1}=2,n_{2}=1,n_{3}=1$ and $n_{4}=3$.}
\label{fig:diag}
\end{figure}
The initial objects in the recursion are 
\begin{align}
\mathord{\stackrel{{\bf n}_{0}}{\mathbb{N}}}=\mathbb{N},\quad
S_{{\bf n}_{0}}=S,\quad
E_{{\bf n}_{0}}=E,\quad
L_{{\bf n}_{0}}=L.
\end{align} 
\begin{definition}
For $k>0$, assume knowledge of $\mathord{\stackrel{{\bf n}_{k-1}}{\mathbb{N}}},S_{{\bf n}_{k-1}}, E_{{\bf n}_{k-1}},L_{{\bf n}_{k-1}}$.  Define the following sets
\begin{equation}
\mathord{\stackrel{{\bf n}_{k}}{\mathbb{N}}}
=\left\{\begin{array}{ll}
\mathbb{N},&\text{if}\ k=0;\\
\mathord{\stackrel{{\bf n}_{k-1}}{\mathbb{N}}},&\text{if}\ k> 0\ \text{and}\ n_{k}=0;\\
E_{{\bf n}_{k-1}}(\mathord{\stackrel{{\bf n}_{k}-1}{\mathbb{N}}}),&\text{otherwise}.
\end{array}\right. 
\end{equation}
\end{definition}
\begin{remark}
The function $E_{{\bf n}_{k-1}}$ generates a descending filtration in $\mathord{\stackrel{{\bf n}_{k}}{\mathbb{N}}}$, meaning $\mathord{\stackrel{{\bf n}_{k}+1}{\mathbb{N}}}\subset\mathord{\stackrel{{\bf n}_{k}}{\mathbb{N}}}$.  Further note that $\mathord{\stackrel{{\bf n}_{k}}{\mathbb{N}}}\subset\mathord{\stackrel{{\bf n}_{k-1}}{\mathbb{N}}}$.
\end{remark}
\begin{definition}
Define the successor function in $\mathord{\stackrel{{\bf n}_{k}}{\mathbb{N}}}$ as the recursive function
\begin{equation}
S_{{\bf n}_{k}}:\mathord{\stackrel{{\bf n}_{k}}{\mathbb{N}}}\rightarrow\mathord{\stackrel{{\bf n}_{k}}{\mathbb{N}}}\setminus\{\mathord{\stackrel{{\bf n}_{k}}{0}}\}:S_{{\bf n}_{k}}(a)
=\left\{\begin{array}{ll}
S(a),&\text{if}\ k=0;\\
S_{{\bf n}_{k-1}}(a),&\text{if}\ k>0\ \text{and}\ n_{k}=0;\\
E_{{\bf n}_{k-1}}(S_{{\bf n}_{k}-1}(L_{{\bf n}_{k-1}}(a))),&\text{otherwise}.
\end{array}\right. 
\end{equation}
Denote its inverse function as $S_{{\bf n}_{k}}^{-1}$. 
\end{definition}
\begin{definition}\label{def:hyp:gen-2}
 The sequence of regular hyperoperations generated from the 
 succesor function in $\mathord{\stackrel{{\bf n}_{k}}{\mathbb{N}}}$ is the sequence of binary operations $[{\bf n}_{k}:m]:(\mathord{\stackrel{{\bf n}_{k}}{\mathbb{N}}})^{2}\rightarrow\mathord{\stackrel{{\bf n}_{k}}{\mathbb{N}}}$ for $m\in\mathbb{N}$, defined recursively as
 \begin{equation}
 a\:[{\bf n}_{k}:m]\:b=\left\{\begin{array}{ll}
 S_{{\bf n}_{k}}(b),&\text{if}\ m=0;\\ 
 a,&\text{if}\ m=1\ \text{and}\ b=\mathord{\stackrel{{\bf n}_{k}}{0}};\\
 \mathord{\stackrel{{\bf n}_{k}}{0}},&\text{if}\ m=2\ \text{and}\ b=\mathord{\stackrel{{\bf n}_{k}}{0}};\\
 \mathord{\stackrel{{\bf n}_{k}}{1}},&\text{if}\ m\geq 3\ \text{and}\ b=\mathord{\stackrel{{\bf n}_{k}}{0}};\\
 a\:[{\bf n}_{k}:m-1]\:(a\:[{\bf n}_{k}:m]\:S_{{\bf n}_{k}}^{-1}(b)),&\text{otherwise}.
 \end{array}\right.
 \end{equation}
 Note that if $n_{k}=0$, then $[{\bf n}_{k}:m]=[{\bf n}_{k-1}:m]$.
\end{definition}
\begin{remark}
 From Definitions~\ref{defn:hyp} and \ref{def:liverpool} one obtains that $[{\bf n}_{0}:m]=H_{m}$ and $[{\bf n}_{1}:m]=[n:m]$, where ${\bf n}_{1}=n$.
\end{remark}
\begin{definition}\label{def:Ebold-n}
Define the function $E_{{\bf n}_{k}}:\mathord{\stackrel{{\bf n}_{k}}{\mathbb{N}}}\rightarrow\mathord{\stackrel{{\bf n}_{k}}{\mathbb{N}}}:E_{{\bf n}_{k}}(a)=w_{k}\:[{\bf n}_{k}:3]\:a$, for a given base $w_{k}\in\mathord{\stackrel{{\bf n}_{k}}{\mathbb{N}}}\setminus\{\mathord{\stackrel{{\bf n}_{k}}{0}},\mathord{\stackrel{{\bf n}_{k}}{1}}\}$.  Denote its inverse function as $L_{{\bf n}_{k}}$.
\end{definition}
\begin{remark}
The tuple $(\mathord{\stackrel{{\bf n}_{k}}{\mathbb{N}}},[{\bf n}_{k}:1],[{\bf n}_{k}:2])$ is a semiring with additive identity $\mathord{\stackrel{{\bf n}_{k}}{0}}$ and multiplicative identity $\mathord{\stackrel{{\bf n}_{k}}{1}}$.  The semiring isomorphism from $(\mathord{\stackrel{{\bf n}_{0}}{\mathbb{N}}},[{\bf n}_{0}:1],[{\bf n}_{0}:2])$ to $(\mathord{\stackrel{{\bf n}_{k}}{\mathbb{N}}},[{\bf n}_{k}:1],[{\bf n}_{k}:2])$ takes the form
\begin{equation}
E_{{\bf n}_{k-1}}^{n_{k}}\circ\dotsb\circ E_{{\bf n}_{1}}^{n_{2}}\circ E_{{\bf n}_{0}}^{n_{1}},
\end{equation}
where $\circ$ stands for function composition and $E_{{\bf n}_{j-1}}^{n_{j}}$ stands for \textit{apply $n_{j}$-times the function $E_{{\bf n}_{j-1}}$}.
\end{remark}
The following propositions involve relations between hyperoperations with index ${\bf n}_{k}$ and ${\bf n}_{k-1}$.
\begin{proposition}\label{prop:canada}
Let  $a,b\in\mathord{\stackrel{{\bf n}_{k}}{\mathbb{N}}}$, then
\begin{equation}
a\:[{\bf n}_{k}:m]\:b=E^{n_{k}}_{{\bf n}_{k-1}}(L^{n_{k}}_{{\bf n}_{k-1}}(a)\:[{\bf n}_{k-1}:m]\:L^{n_{k}}_{{\bf n}_{k-1}}(b)).
\end{equation}
\end{proposition}
\begin{proof}
Let $a,b\in\mathord{\stackrel{{\bf n}_{k}}{\mathbb{N}}}$.  In analogy to Theorem~\ref{th:oxford} one obtains
\begin{equation}
a\:[{\bf n}_{k}:m]\:b=E_{{\bf n}_{k-1}}(L_{{\bf n}_{k-1}}(a)\:[{\bf n}_{k}-1:m]\:L_{{\bf n}_{k-1}}(b)).
\end{equation}
Similarly, one obtains
\begin{equation}
L_{{\bf n}_{k-1}}(a)\:[{\bf n}_{k}-1:m]\:L_{{\bf n}_{k-1}}(b)=E_{{\bf n}_{k-1}}(L^{2}_{{\bf n}_{k-1}}(a)\:[{\bf n}_{k}-2:m]\:L^{2}_{{\bf n}_{k-1}}(b)),
\end{equation}
therefore,
\begin{equation}
a\:[{\bf n}_{k}:m]\:b=E^{2}_{{\bf n}_{k-1}}(L^{2}_{{\bf n}_{k-1}}(a)\:[{\bf n}_{k}-2:m]\:L^{2}_{{\bf n}_{k-1}}(b)).
\end{equation}
Repeating this procedure, one obtains
\begin{subequations}
\begin{align}
a\:[{\bf n}_{k}:m]\:b&=E^{n_{k}}_{{\bf n}_{k-1}}(L^{n_{k}}_{{\bf n}_{k-1}}(a)\:[{\bf n}_{k}-n_{k}:m]\:L^{n_{k}}_{{\bf n}_{k-1}}(b))\\
&=E^{n_{k}}_{{\bf n}_{k-1}}(L^{n_{k}}_{{\bf n}_{k-1}}(a)\:[({\bf n}_{k-1},0):m]\:L^{n_{k}}_{{\bf n}_{k-1}}(b))\\
&=E^{n_{k}}_{{\bf n}_{k-1}}(L^{n_{k}}_{{\bf n}_{k-1}}(a)\:[{\bf n}_{k-1}:m]\:L^{n_{k}}_{{\bf n}_{k-1}}(b)).
\end{align}
\end{subequations}
\end{proof}
\begin{proposition}\label{prop:august}
Let  $a\in\mathord{\stackrel{{\bf n}_{k}}{\mathbb{N}}}$.  If $w_{k}=w_{k-1}\:[{\bf n}_{k}:4]\:(n_{k}+1)$, then
\begin{equation}
E_{{\bf n}_{k}}(a)=E_{{\bf n}_{k-1}}(a).
\end{equation}
\end{proposition}
\begin{proof}
Let  $a\in\mathord{\stackrel{{\bf n}_{k}}{\mathbb{N}}}$.  From Proposition~\ref{prop:canada} one obtains
\begin{subequations}
\begin{align}
E_{{\bf n}_{k}}(a)&=w_{k}\:[{\bf n}_{k}:3]\:a\\
&=E^{n_{k}}_{{\bf n}_{k-1}}(L^{n_{k}}_{{\bf n}_{k-1}}(w_{k})\:[{\bf n}_{k-1}:3]\:L^{n_{k}}_{{\bf n}_{k-1}}(a)).
\end{align}
\end{subequations}
If $w_{k}=w_{k-1}\:[{\bf n}_{k}:4]\:(n_{k}+1)$, then
\begin{subequations}
\begin{align}
E_{{\bf n}_{k}}(a)&=E^{n_{k}}_{{\bf n}_{k-1}}(w_{k-1}\:[{\bf n}_{k-1}:3]\:L^{n_{k}}_{{\bf n}_{k-1}}(a))\\
&=E^{n_{k}}_{{\bf n}_{k-1}}(E_{{\bf n}_{k-1}}(L^{n_{k}}_{{\bf n}_{k-1}}(a)))\\
&=E_{{\bf n}_{k-1}}(a).
\end{align}
\end{subequations}
\end{proof}
\begin{proposition}
Let  $a,b\in\mathord{\stackrel{{\bf n}_{k}}{\mathbb{N}}}$.  If $w_{k-1}=w_{k-2}\:[{\bf n}_{k-1}:4]\:(n_{k-1}+1)$, then
\begin{equation}
\mathord{\stackrel{{\bf n}_{k}}{\mathbb{N}}}=\mathord{\stackrel{{\bf n}_{k-1}+n_{k}}{\mathbb{N}}}\quad
\text{and}\quad
a\:[{\bf n}_{k}:m]\:b=a\:[{\bf n}_{k-1}+n_{k}:m]\:b.
\end{equation}
\end{proposition}
\begin{proof}
Note that $\mathord{\stackrel{{\bf n}_{k}}{\mathbb{N}}}=E^{n_{k}}_{{\bf n}_{k-1}}(\mathord{\stackrel{({\bf n}_{k-1},0)}{\mathbb{N}}})=E^{n_{k}}_{{\bf n}_{k-1}}(\mathord{\stackrel{{\bf n}_{k-1}}{\mathbb{N}}})$ and $\mathord{\stackrel{{\bf n}_{k-1}+n_{k}}{\mathbb{N}}}=E^{n_{k}}_{{\bf n}_{k-2}}(\mathord{\stackrel{{\bf n}_{k-1}}{\mathbb{N}}})$.  Let $w_{k-1}=w_{k-2}\:[{\bf n}_{k-1}:4]\:(n_{k-1}+1)$.  From Proposition~\ref{prop:august} one obtains $E^{n_{k}}_{{\bf n}_{k-1}}=E^{n_{k}}_{{\bf n}_{k-2}}$, therefore $\mathord{\stackrel{{\bf n}_{k}}{\mathbb{N}}}=\mathord{\stackrel{{\bf n}_{k-1}+n_{k}}{\mathbb{N}}}$.  Finally, from Propositions~\ref{prop:canada} and \ref{prop:august} one obtains
\begin{subequations}
\begin{align}
a\:[{\bf n}_{k}:m]\:b&=E^{n_{k}}_{{\bf n}_{k-1}}(L^{n_{k}}_{{\bf n}_{k-1}}(a)\:[{\bf n}_{k-1}:m]\:L^{n_{k}}_{{\bf n}_{k-1}}(b))\\
&=E^{n_{k}}_{{\bf n}_{k-2}}(L^{n_{k}}_{{\bf n}_{k-2}}(a)\:[{\bf n}_{k-1}:m]\:L^{n_{k}}_{{\bf n}_{k-2}}(b))\\
&=a\:[{\bf n}_{k-1}+n_{k}:m]\:b.
\end{align}
\end{subequations}
\end{proof}

%% file: 3-rep-numb.tex
%
A recursive representation of nonnegative integers using commutative hyperoperations can be obtained from its power series expansion.  Let $a,b\in\mathbb{N}$, the reader is reminded of ordinary notation $a+b=H_{1}(a,b)$, $a\times b=H_{2}(a,b)$ and $a^{b}=H_{3}(a,b)$. It is well known that any integer $a\in\mathbb{N}$ can be represented uniquely as a power series base $w$ in the form
\begin{equation}\label{eq:pwseries-0}
a=d_{0}+d_{1}\times w+d_{2}\times w^{2}+\dotsb+d_{m}\times w^{m},
\end{equation}
where the coefficients $d_{k}\in\{0,1,2,\dotsc,S^{-1}(w)\}$ are the digits and the maximum exponent $m\in\mathbb{N}$ is such that $d_{k>m}=0$.  In the power series of Eq.~\eqref{eq:pwseries-0} one can write each exponent as a power series base $w$ and proceed iteratively.  This is known as the hereditary representation of a natural number.  Let $w$ be the base of the exponential function $E$, underlying the definition of the commutative hyperoperations $\{F_{n}\}_{n\in\mathbb{N}}$.  This power series can be written recursively as
\begin{equation}\label{eq:pwseries-1}
a_{k}=\left\{\begin{array}{ll}
d_{0},& \text{if}\ k=0;\\
F_{0}(F_{1}(d_{k},\mathord{\stackrel{1}{k}}),a_{k-1}),& \text{if}\ 0<k\leq m.
\end{array}\right.
\end{equation}
The term $\mathord{\stackrel{1}{k}}$ is the scale of the digit $d_{k}$ and $a_{m}=a$. Now consider the coefficients $e_{k}\in\{\mathord{\stackrel{n}{0}},\mathord{\stackrel{n}{1}},\mathord{\stackrel{n}{0}},\dotsc,S_{n}^{-1}(\mathord{\stackrel{n}{w}})\}$ for $0\leq k\leq m$.  Applying the function $E^{n}$ to both sides of Eq.~\eqref{eq:pwseries-1} one arrives at the more general statement that any element $b\in\mathord{\stackrel{n}{\mathbb{N}}}$ can be represented as a power series in the semiring $(\mathord{\stackrel{n}{\mathbb{N}}},F_{n},F_{n+1})$, taking the form
\begin{equation}\label{eq:pwseries-2}
b_{k}=\left\{\begin{array}{ll}
e_{0},& \text{if}\ k=0;\\
F_{n}(F_{n+1}(e_{k},\mathord{\stackrel{n+1}{k}}),b_{k-1}),& \text{if}\ 0<k\leq m.
\end{array}\right.
\end{equation} 
The term $\mathord{\stackrel{n+1}{k}}$ is the \textit{scale} of the \text{digit} $e_{k}$ and $b_{m}=b$.  The Eq.~\eqref{eq:pwseries-2} allows hereditary representations written in terms of commutative hyperoperations.  The scale $\mathord{\stackrel{1}{k}}\in\mathord{\stackrel{1}{\mathbb{N}}}$, can be represented as a power series in the semiring $(\mathord{\stackrel{1}{\mathbb{N}}},F_{1},F_{2})$ and so on.  The procedure repeats iteratively until all scales are expanded.  For example, the power series expansion of $266$ base $w=3$ is
\begin{equation}
266=2\times 3^{0}+1\times 3^{1}+2\times 3^{2}+1\times 3^{5}.
\end{equation}
In the usual hereditary representation, the exponent $k=5$ is replaced by its power series expansion $k=2\times 3^{0}+1\times 3^{1}$. Alternatively, using Eq.~\eqref{eq:pwseries-1} with $a=266\in\mathord{\stackrel{0}{\mathbb{N}}}$ and $w=3$, one obtains
\begin{subequations}\label{eq:nasa}
\begin{align}
a_{0}&=\mathord{\stackrel{0}{2}},\\
a_{1}&=F_{0}(F_{1}(\mathord{\stackrel{0}{1}},\mathord{\stackrel{1}{1}}),a_{0}),\\
a_{2}&=F_{0}(F_{1}(\mathord{\stackrel{0}{2}},\mathord{\stackrel{1}{2}}),a_{1}),\\
a_{3}&=a_{2},\\
a_{4}&=a_{3},\\
a_{5}&=F_{0}(F_{1}(\mathord{\stackrel{0}{1}},\mathord{\stackrel{1}{5}}),a_{4})=a.
\end{align}
\end{subequations}
According to the hereditary representation, the scale $\mathord{\stackrel{1}{5}}$ in $a_{5}$ from Eq.~\eqref{eq:nasa}, must be expanded.  From Eq.~\eqref{eq:pwseries-2} with $b=\mathord{\stackrel{1}{5}}\in\mathord{\stackrel{1}{\mathbb{N}}}$ and $w=3$, one obtains
\begin{subequations}
\begin{align}
b_{0}&=\mathord{\stackrel{1}{2}},\\
b_{1}&=F_{1}(F_{2}(\mathord{\stackrel{1}{1}},\mathord{\stackrel{2}{1}}),b_{0})=b.
\end{align}
\end{subequations}
There is no further iteration. The expansion of $266$ base $w=2$ is
\begin{equation}
266=2^{1}+2^{3}+2^{8}.
\end{equation}
From Eq.~\eqref{eq:pwseries-1} with $a=266$ and $w=2$, one obtains
\begin{subequations}\label{eq:spacex}
\begin{align}
a_{0}&=\mathord{\stackrel{0}{0}},\\
a_{1}&=F_{0}(F_{1}(\mathord{\stackrel{0}{1}},\mathord{\stackrel{1}{1}}),a_{0}),\\
a_{2}&=a_{1},\\
a_{3}&=F_{0}(F_{1}(\mathord{\stackrel{0}{1}},\mathord{\stackrel{1}{3}}),a_{2}),\\
a_{4}&=a_{3},\\
&\vdots\nonumber\\
a_{7}&=a_{6},\\
a_{8}&=F_{0}(F_{1}(\mathord{\stackrel{0}{1}},\mathord{\stackrel{1}{8}}),a_{7})=a.
\end{align}
\end{subequations}
The scale $\mathord{\stackrel{1}{3}}$ in $a_{3}$ from Eq.~\eqref{eq:spacex}, must be expanded.  From Eq.~\eqref{eq:pwseries-2} with $b=\mathord{\stackrel{1}{3}}\in\mathord{\stackrel{1}{\mathbb{N}}}$ and $w=2$, one obtains
\begin{subequations}
\begin{align}
b_{0}&=\mathord{\stackrel{1}{1}},\\
b_{1}&=F_{1}(F_{2}(\mathord{\stackrel{1}{1}},\mathord{\stackrel{2}{1}}),b_{0})=b.
\end{align}
\end{subequations}
There is no further iteration in $a_{3}$ from Eq.~\eqref{eq:spacex}.  The scale $\mathord{\stackrel{1}{8}}$ in $a_{8}$ from Eq.~\eqref{eq:spacex}, must be expanded.  Note that $\mathord{\stackrel{1}{8}}=2^{8}=2^{2^{3}}$.  Let $c=\mathord{\stackrel{1}{8}}\in\mathord{\stackrel{1}{\mathbb{N}}}$ and $w=2$.  Using the same procedure as in Eq.~\eqref{eq:pwseries-2}, one obtains
\begin{subequations}\label{eq:solaris}
\begin{align}
c_{0}&=\mathord{\stackrel{1}{0}},\\
c_{1}&=c_{0},\\
c_{2}&=c_{1},\\
c_{3}&=F_{1}(F_{2}(\mathord{\stackrel{1}{1}},\mathord{\stackrel{2}{3}}),c_{2})=c.
\end{align}
\end{subequations}
The scale $\mathord{\stackrel{2}{3}}$ in $c_{3}$ from Eq.~\eqref{eq:solaris}, must be expanded.  Note that $\mathord{\stackrel{2}{3}}=2^{2^{3}}=2^{2^{1+2}}$.  Let $c'=\mathord{\stackrel{2}{3}}\in\mathord{\stackrel{2}{\mathbb{N}}}$ with $w=2$.  Using the same procedure as in Eq.~\eqref{eq:pwseries-2}, one obtains
\begin{subequations}
\begin{align}
c'_{0}&=\mathord{\stackrel{2}{1}},\\
c'_{1}&=F_{2}(F_{3}(\mathord{\stackrel{2}{1}},\mathord{\stackrel{3}{1}}),c'_{0})=c'.
\end{align}
\end{subequations}
There is no further iteration.

%% file: 4-fields.tex
This section presents the standard construction of the rational numbers \cite{WA65} as a parallel construction departing from the initial sets $\mathord{\stackrel{n}{\mathbb{N}}}$ for $n\in\mathbb{N}$.  The procedure of inverse completion is applied to each commutative monoid $(\mathord{\stackrel{n}{\mathbb{N}}},F_{n})$, obtaining the abelian group $(\mathbb{Z}_{n,\square},\oplus_{n,\square})$.  A multiplication $\otimes_{n,\square}$ is defined in $\mathbb{Z}_{n,\square}$, making the tuple $(\mathbb{Z}_{n,\square},\oplus_{n,\square},\otimes_{n,\square})$ an integral domain.  After applying inverse completion to the latter, it is obtained the field of quotients $(\mathbb{Q}_{n,\boxtimes},\oplus_{n,\boxtimes},\otimes_{n,\boxtimes})$.

%% file: 5-fields.tex
\begin{definition}\label{def:goku}
Define the cartesian product $(\mathord{\stackrel{0}{\mathbb{N}}}\times\mathord{\stackrel{0}{\mathbb{N}}},\oplus_{0})$ such that
\begin{equation}
\forall(x,y),(u,v)\in\mathord{\stackrel{0}{\mathbb{N}}}\times\mathord{\stackrel{0}{\mathbb{N}}}:(x,y)\oplus_{0}(u,v)=(F_{0}(x,u),F_{0}(y,v)).
\end{equation}
\end{definition}
\begin{definition}\label{def:robin}
Define the function $\vec{E}$ in $\mathord{\stackrel{0}{\mathbb{N}}}\times\mathord{\stackrel{0}{\mathbb{N}}}$ as
\begin{equation}
\vec{E}(x,y)=(E(x),E(y)).
\end{equation}
Denote by $\vec{L}$ its inverse function.
\end{definition}
\begin{remark}
Since $\mathord{\stackrel{n}{\mathbb{N}}}\subset\mathord{\stackrel{n-1}{\mathbb{N}}}$, then $\mathord{\stackrel{n}{\mathbb{N}}}\times\mathord{\stackrel{n}{\mathbb{N}}}\subset\mathord{\stackrel{n-1}{\mathbb{N}}}\times\mathord{\stackrel{n-1}{\mathbb{N}}}$.
\end{remark}
\begin{definition}
Define the binary operation
\begin{gather}
\oplus_{n}:(\mathord{\stackrel{n}{\mathbb{N}}}\times\mathord{\stackrel{n}{\mathbb{N}}})^{2}\rightarrow\mathord{\stackrel{n}{\mathbb{N}}}\times\mathord{\stackrel{n}{\mathbb{N}}}:\nonumber\\
(x,y)\oplus_{n}(u,v)=(F_{n}(x,u),F_{n}(y,v)).
\end{gather}
Note that
\begin{equation}\label{eq:toby}
(x,y)\oplus_{n}(u,v)=\vec{E}(\vec{L}(x,y)\oplus_{n-1}\vec{L}(u,v)).
\end{equation}
\end{definition}
\begin{proposition}
Let $(x,y),(u,v)\in\mathord{\stackrel{n}{\mathbb{N}}}\times\mathord{\stackrel{n}{\mathbb{N}}}$.  The relation $\stackrel{n}{\square}$ defined as
\begin{equation}
(x,y)\stackrel{n}{\square}(u,v)\iff F_{n}(x,v)=F_{n}(y,u),
\end{equation}
is a congruence relation on the commutative monoid $(\mathord{\stackrel{n}{\mathbb{N}}}\times\mathord{\stackrel{n}{\mathbb{N}}},\oplus_{n})$.
\end{proposition}
\begin{proof}
Reflexivity and symmetry are straightforward.  It remains to prove transitivity and compatibility:
\begin{enumerate}
\item \textit{Transitivity.} Let $(x_{i},y_{i})\in\mathord{\stackrel{n}{\mathbb{N}}}\times\mathord{\stackrel{n}{\mathbb{N}}}$ for $i=1,2,3$, such that
\begin{equation}
(x_{1},y_{1})\stackrel{n}{\square}(x_{2},y_{2})\quad
\text{and}\quad
(x_{2},y_{2})\stackrel{n}{\square}(x_{3},y_{3}).
\end{equation}
Since
\begin{equation}\label{eq:vargas}
F_{n}(x_{1},y_{2})=F_{n}(x_{2},y_{1})\quad
\text{and}\quad
F_{n}(x_{2},y_{3})=F_{n}(x_{3},y_{2}),
\end{equation}
adding both equations one obtains
\begin{equation}
F_{n}(x_{1},y_{3})=F_{n}(x_{3},y_{1}).
\end{equation}
That is $(x_{1},y_{1})\stackrel{n}{\square}(x_{3},y_{3})$.  
\item \textit{Compatibility.} Let $(x_{i},y_{i}),(u_{i},v_{i})\in\mathord{\stackrel{n}{\mathbb{N}}}\times\mathord{\stackrel{n}{\mathbb{N}}}$ for $i=1,2$, such that
\begin{equation}\label{eq:pontiac}
(x_{1},y_{1})\stackrel{n}{\square}(x_{2},y_{2})\quad
\text{and}\quad
(u_{1},v_{1})\stackrel{n}{\square}(u_{2},v_{2}).
\end{equation}
For $i=1,2$, denote
\begin{equation}
(a_{i},b_{i})=(x_{i},y_{i})\oplus_{n}(u_{i},v_{i})=(F_{n}(x_{i},u_{i}),F_{n}(y_{i},v_{i})).
\end{equation}
From Eq.~\eqref{eq:pontiac} it follows that
\begin{equation}
F_{n}(a_{1},b_{2})=F_{n}(a_{2},b_{1}),
\end{equation}
implying $(a_{1},b_{1})\stackrel{n}{\square}(a_{2},b_{2})$.
\end{enumerate}
\end{proof}
\begin{notation}
Consider the quotient structure
\begin{equation}
(\mathbb{Z}_{n,\square},\oplus_{n,\square})=\left(\frac{\mathord{\stackrel{n}{\mathbb{N}}}\times\mathord{\stackrel{n}{\mathbb{N}}}}{\stackrel{n}{\square}},\oplus_{n,\square}\right),
\end{equation}
where $\oplus_{n,\square}$ is the operation induced on $(\mathord{\stackrel{n}{\mathbb{N}}}\times\mathord{\stackrel{n}{\mathbb{N}}})/\!\stackrel{n}{\square}$ by $\oplus_{n}$.  Denote by $[\![(x,y)]\!]_{n,\square}$ the equivalence class of $(x,y)$ under $\stackrel{n}{\square}$.
\end{notation}
\begin{remark}
For all $(x,y),(u,v)\in\mathord{\stackrel{n}{\mathbb{N}}}\times\mathord{\stackrel{n}{\mathbb{N}}}$, if $(x,y)\stackrel{n}{\square}(u,v)$, then $(y,x)\stackrel{n}{\square}(v,u)$.
\end{remark}
\begin{definition}
Define the involution
\begin{gather}
T_{n,\square}\: :\mathbb{Z}_{n,\square}\rightarrow\mathbb{Z}_{n,\square}:\nonumber\\
T_{n,\square}\:[\![(x,y)]\!]_{n,\square}=[\![(y,x)]\!]_{n,\square}.
\end{gather}
\end{definition}
\begin{proposition}\label{prop:kamikaze}
The tuple $(\mathbb{Z}_{n,\square},\oplus_{n,\square})$ is an abelian group.
\end{proposition}
\begin{proof}  
Since $\stackrel{n}{\square}$ is a congruence relation on the commutative monoid $(\mathord{\stackrel{n}{\mathbb{N}}}\times\mathord{\stackrel{n}{\mathbb{N}}},\oplus_{n})$, the corresponding quotient structure $(\mathbb{Z}_{n,\square},\oplus_{n,\square})$ is also a commutative monoid.  It is straightforward to verify that the involution $T_{n,\square}$ is the group inverse in $(\mathbb{Z}_{n,\square},\oplus_{n,\square})$, therefore it is a group.
\end{proof}
\begin{definition}\label{def:peterpan}
The ordering $\stackrel{n}{<}$ in $\mathbb{Z}_{n,\square}$ is defined as
\begin{equation}
[\![(x,y)]\!]_{n,\square}\stackrel{n}{<}[\![(u,v)]\!]_{n,\square}\iff F_{n}(x,v)<F_{n}(y,u).
\end{equation} 
\end{definition}
\begin{proposition}\label{prop:kiev}
Let $(x,y),(u,v)\in\mathord{\stackrel{n}{\mathbb{N}}}\times\mathord{\stackrel{n}{\mathbb{N}}}$, then
\begin{equation}
(x,y)\stackrel{n}{\square}(u,v)\iff\vec{E}((x,y))\stackrel{n+1}{\square}\Vec{E}((u,v)).
\end{equation}
\end{proposition}
\begin{proof} 
The proof that $(x,y)\stackrel{n}{\square}(u,v)\Rightarrow\vec{E}(x,y)\stackrel{n+1}{\square}\vec{E}(u,v)$ is the following: If $(x,y)\stackrel{n}{\square}(u,v)$, then $F_{n}(x,v)=F_{n}(y,u)$.  Applying the exponential $E$ at both sides, one obtains $F_{n+1}(E(x),E(v))=F_{n+1}(E(y),E(u))$.  From the definition of $\stackrel{n+1}{\square}$, this implies $(E(x),E(y))\stackrel{n+1}{\square}(E(u),E(v))$, that is, $\vec{E}(x,y)\stackrel{n+1}{\square}\vec{E}(u,v)$.  Noting that $E$ is an injective function, the converse statement is proven in a similar way. 
\end{proof}
\begin{definition}\label{def:brexit}
Define the function $E_{\square}$ in $U_{\mathbb{Z}}=\bigcup_{n\in\mathbb{N}}\mathbb{Z}_{n,\square}$ as
\begin{equation}
E_{\square}([\![a]\!]_{n,\square})=[\![\vec{E}(a)]\!]_{n+1,\square}.
\end{equation}
Denote by $L_{\square}$ its inverse function.
\end{definition}
\begin{remark}
Note that $\mathbb{Z}_{n,\square}=E_{\square}(\mathbb{Z}_{n-1,\square})$.
\end{remark}
\begin{proposition}\label{prop:piscis}
Let $a,b\in\mathbb{Z}_{n,\square}$, then
\begin{equation}
a\oplus_{n,\square}b=E_{\square}(L_{\square}(a)\oplus_{n-1,\square}L_{\square}(b)).
\end{equation}
\end{proposition}
\begin{proof}
Let $a=[\![(x,y)]\!]_{n,\square}$ and $b=[\![(u,v)]\!]_{n,\square}$, then
\begin{subequations}
\begin{align}
[\![(x,y)]\!]_{n,\square}\oplus_{n,\square}[\![(u,v)]\!]_{n,\square}
&=[\![(x,y)\oplus_{n}(u,v)]\!]_{n,\square}\\
&=[\![\vec{E}(\vec{L}(x,y)\oplus_{n-1}\vec{L}(u,v))]\!]_{n,\square}\\
&=E_{\square}([\![\vec{L}(x,y)\oplus_{n-1}\vec{L}(u,v)]\!]_{n-1,\square})\\
&=E_{\square}([\![\vec{L}(x,y)]\!]_{n-1,\square}\oplus_{n-1,\square}[\![\vec{L}(u,v)]\!]_{n-1,\square})\\
&=E_{\square}(L_{\square}([\![(x,y)]\!]_{n,\square})\oplus_{n-1,\square}L_{\square}([\![(u,v)]\!]_{n,\square})).
\end{align}
\end{subequations}
\end{proof}
\begin{notation}
Denote $k_{n,\square}=[\![(\mathord{\stackrel{n}{k}},\mathord{\stackrel{n}{0}})]\!]_{n,\square}$ and $\mathbb{N}_{n,\square}=[\![(\mathord{\stackrel{n}{\mathbb{N}}},\mathord{\stackrel{n}{0}})]\!]_{n,\square}$.
\end{notation}
\begin{remark}
Every element in the abelian group $(\mathbb{Z}_{n,\square},\oplus_{n,\square})$ can be written as $a\oplus_{n,\square}(T_{n,\square}\:b)$, where $a,b\in\mathbb{N}_{n,\square}$.
\end{remark}

%% file: 6-fields.tex
\begin{definition}\label{def:biontech}
Define the binary operation
\begin{gather}
\otimes_{n}:(\mathord{\stackrel{n}{\mathbb{N}}}\times\mathord{\stackrel{n}{\mathbb{N}}})^{2}\rightarrow\mathord{\stackrel{n}{\mathbb{N}}}\times\mathord{\stackrel{n}{\mathbb{N}}}:\nonumber\\
(x,y)\otimes_{n}(u,v)=(F_{n}(F_{n+1}(x,u),F_{n+1}(y,v)),F_{n}(F_{n+1}(y,u),F_{n+1}(x,v))).
\end{gather}
\end{definition}
\begin{proposition}
Let $(x_{i},y_{i}),(u_{i},v_{i})\in\mathord{\stackrel{n}{\mathbb{N}}}\times\mathord{\stackrel{n}{\mathbb{N}}}$ for $i=1,2$, such that
\begin{equation}\label{eq:impala}
(x_{1},y_{1})\stackrel{n}{\square}(x_{2},y_{2})\quad
\text{and}\quad
(u_{1},v_{1})\stackrel{n}{\square}(u_{2},v_{2}),
\end{equation}
then
\begin{equation}
((x_{1},y_{1})\otimes_{n}(u_{1},v_{1}))\stackrel{n}{\square}((x_{2},y_{2})\otimes_{n}(u_{2},v_{2})).
\end{equation}
\end{proposition}
\begin{proof}
For $i=1,2$, denote
\begin{equation}
(a_{i},b_{i})=((x_{i},y_{i})\otimes_{n}(u_{i},v_{i}))
=(F_{n}(F_{n+1}(x_{i},u_{i}),F_{n+1}(y_{i},v_{i})),F_{n}(F_{n+1}(y_{i},u_{i}),F_{n+1}(x_{i},v_{i}))).
\end{equation}
From Eq.~\eqref{eq:impala} it follows that
\begin{equation}
F_{n}(a_{1},b_{2})=F_{n}(b_{1},a_{2}),
\end{equation}
therefore
\begin{equation}
(a_{1},b_{1})\stackrel{n}{\square}(a_{2},b_{2}).
\end{equation}
\end{proof}
\begin{definition}
Define the binary operation
\begin{gather}
\otimes_{n,\square}:(\mathbb{Z}_{n,\square})^{2}\rightarrow\mathbb{Z}_{n,\square}:\nonumber\\
[\![(x,y)]\!]_{n,\square}\otimes_{n,\square}[\![(u,v)]\!]_{n,\square}=[\![(x,y)\otimes_{n}(u,v)]\!]_{n,\square}.
\end{gather}
\end{definition}
\begin{proposition}\label{prop:aries}
Let $a,b\in\mathbb{Z}_{n,\square}$, then
\begin{equation}
a\otimes_{n,\square}b=E_{\square}(L_{\square}(a)\otimes_{n-1,\square}L_{\square}(b)).
\end{equation}
\end{proposition}
\begin{proof}
Let $a=[\![(x,y)]\!]_{n,\square}$ and $b=[\![(u,v)]\!]_{n,\square}$, then
\begin{subequations}
\begin{gather}
[\![(x,y)]\!]_{n,\square}\otimes_{n,\square}[\![(u,v)]\!]_{n,\square}\\
=[\![(F_{n}(F_{n+1}(x,u),F_{n+1}(y,v)),F_{n}(F_{n+1}(x,v),F_{n+1}(y,u)))]\!]_{n,\square}\\
=E_{\square}([\![(F_{n-1}(F_{n}(L(x),L(u)),F_{n}(L(y),L(v))),F_{n-1}(F_{n}(L(x),L(v)),F_{n}(L(y),L(u))))]\!]_{n-1,\square})\\
=E_{\square}([\![(L(x),L(y))]\!]_{n-1,\square}\otimes_{n-1,\square}[\![(L(u),L(v))]\!]_{n-1,\square})\\
=E_{\square}(L_{\square}([\![(x,y)]\!]_{n,\square})\otimes_{n-1,\square}L_{\square}([\![(u,v)]\!]_{n,\square})).
\end{gather}
\end{subequations}
\end{proof}
\begin{proposition}\label{prop:mario}
The tuple $(\mathbb{N}_{n,\square},\oplus_{n,\square},\otimes_{n,\square})$ is a commutative semiring. 
\end{proposition}
\begin{proof}
The semiring structure in $(\mathbb{N}_{n,\square},\oplus_{n,\square},\otimes_{n,\square})$ is induced by $(\mathord{\stackrel{n}{\mathbb{N}}},F_{n},F_{n+1})$, in account that 
\begin{subequations}
\begin{align}
x_{n,\square}\oplus_{n,\square}y_{n,\square}&=[\![(F_{n}(x,y),\mathord{\stackrel{n}{0}})]\!]_{n,\square},\\
x_{n,\square}\otimes_{n,\square}y_{n,\square}&=[\![(F_{n+1}(x,y),\mathord{\stackrel{n}{0}})]\!]_{n,\square}.
\end{align}
\end{subequations}
\end{proof}
\begin{remark}
The additive inverse $T_{n,\square}$ in $(\mathbb{Z}_{n,\square},\oplus_{n,\square},\otimes_{n,\square})$, exhibits the following properties:
\begin{subequations}\label{eq:ties-argh}
\begin{gather} 
T_{n,\square}\:(T_{n,\square}\:a)=a,\label{ties-argh:1}\\
(T_{n,\square}\:a)\oplus_{n,\square}(T_{n,\square}\:b)=T_{n,\square}\:(a\oplus_{n,\square}b),\label{ties-argh:2}\\
(T_{n,\square}\:a)\otimes_{n,\square}b=a\otimes_{n,\square}(T_{n,\square}\:b)=T_{n,\square}\:(a\otimes_{n,\square}b),\label{ties-argh:3}\\ 
(T_{n,\square}\:a)\otimes_{n,\square}(T_{n,\square}\:b)=a\otimes_{n,\square}b.\label{ties-argh:4}
\end{gather}
\end{subequations}
\end{remark}
\begin{proposition}\label{prop:luigi}
The tuple $(\mathbb{Z}_{n,\square},\otimes_{n,\square})$ is a commutative monoid.
\end{proposition}
\begin{proof}
From Proposition~\ref{prop:mario} the tuple $(\mathbb{N}_{n,\square},\otimes_{n,\square})$ is a commutative monoid.  From the properties in Eq.~\eqref{eq:ties-argh} any computation in $(\mathbb{Z}_{n,\square},\otimes_{n,\square})$ reduces to a computation in $(\mathbb{N}_{n,\square},\otimes_{n,\square})$, up to the overall action of the additive inverse $T_{n,\square}$.
\end{proof}
\begin{theorem}\label{th:taurus}
The tuple $(\mathbb{Z}_{n,\square},\oplus_{n,\square},\otimes_{n,\square})$ is an integral domain.
\end{theorem}
\begin{proof}
From Proposition~\ref{prop:kamikaze} the tuple $(\mathbb{Z}_{n,\square},\oplus_{n,\square})$ is an abelian group and from Proposition~\ref{prop:luigi} the tuple $(\mathbb{Z}_{n,\square},\otimes_{n,\square})$ is a commutative monoid.  The tuple $(\mathbb{Z}_{n,\square},\oplus_{n,\square},\otimes_{n,\square})$ is a commutative ring, if it fulfills distributivity.  Let $a,b,c\in\mathbb{N}_{n,\square}$. It is sufficient to prove the following eight cases:
\begin{subequations}
\begin{align}
a\otimes_{n,\square}(b\oplus_{n,\square}c)&=(a\otimes_{n,\square}b)\oplus_{n,\square}(a\otimes_{n,\square}c),\label{eq:disco-1}\\
a\otimes_{n,\square}(b\oplus_{n,\square}(T_{n,\square}\:c))&=(a\otimes_{n,\square}b)\oplus_{n,\square}(a\otimes_{n,\square}(T_{n,\square}\:c)),\label{eq:disco-2}\\
a\otimes_{n,\square}((T_{n,\square}\:b)\oplus_{n,\square}c)&=(a\otimes_{n,\square}(T_{n,\square}\:b))\oplus_{n,\square}(a\otimes_{n,\square}c),\label{eq:disco-3}\\
a\otimes_{n,\square}((T_{n,\square}\:b)\oplus_{n,\square}(T_{n,\square}\:c))&=(a\otimes_{n,\square}(T_{n,\square}\:b))\oplus_{n,\square}(a\otimes_{n,\square}(T_{n,\square}\:c)),\label{eq:disco-4}\\
(T_{n,\square}\:a)\otimes_{n,\square}(b\oplus_{n,\square}c)&=((T_{n,\square}\:a)\otimes_{n,\square}b)\oplus_{n,\square}((T_{n,\square}\:a)\otimes_{n,\square}c),\label{eq:disco-5}\\
(T_{n,\square}\:a)\otimes_{n,\square}(b\oplus_{n,\square}(T_{n,\square}\:c))&=((T_{n,\square}\:a)\otimes_{n,\square}b)\oplus_{n,\square}((T_{n,\square}\:a)\otimes_{n,\square}(T_{n,\square}\:c)),\label{eq:disco-6}\\
(T_{n,\square}\:a)\otimes_{n,\square}((T_{n,\square}\:b)\oplus_{n,\square}c)&=((T_{n,\square}\:a)\otimes_{n,\square}(T_{n,\square}\:b))\oplus_{n,\square}((T_{n,\square}\:a)\otimes_{n,\square}c),\label{eq:disco-7}\\
(T_{n,\square}\:a)\otimes_{n,\square}((T_{n,\square}\:b)\oplus_{n,\square}(T_{n,\square}\:c))&=((T_{n,\square}\:a)\otimes_{n,\square}(T_{n,\square}\:b))\oplus_{n,\square}((T_{n,\square}\:a)\otimes_{n,\square}(T_{n,\square}\:c)).\label{eq:disco-8}
\end{align}
\end{subequations}
Let $x,y,z\in\mathord{\stackrel{n}{\mathbb{N}}}$ such that $a=[\![(x,\mathord{\stackrel{n}{0}})]\!]_{n,\square},b=[\![(y,\mathord{\stackrel{n}{0}})]\!]_{n,\square},c=[\![(z,\mathord{\stackrel{n}{0}})]\!]_{n,\square}$.  The Eq.~\eqref{eq:disco-2} becomes,
\begin{subequations}
\begin{gather}
[\![(x,\mathord{\stackrel{n}{0}})]\!]_{n,\square}\otimes_{n,\square}([\![(y,\mathord{\stackrel{n}{0}})]\!]_{n,\square}\oplus_{n,\square}[\![(\mathord{\stackrel{n}{0}},z)]\!]_{n,\square})\\
=[\![(x,\mathord{\stackrel{n}{0}})]\!]_{n,\square}\otimes_{n,\square}[\![(y,z)]\!]_{n,\square}\\
=[\![(F_{n+1}(x,y),F_{n+1}(x,z))]\!]_{n,\square}\\
=[\![(F_{n+1}(x,y),\mathord{\stackrel{n}{0}})]\!]_{n,\square}\oplus_{n,\square}[\![(\mathord{\stackrel{n}{0}},F_{n+1}(x,z))]\!]_{n,\square}\\
=([\![(x,\mathord{\stackrel{n}{0}})]\!]_{n,\square}\otimes_{n,\square}[\![(y,\mathord{\stackrel{n}{0}})]\!]_{n,\square})\oplus_{n,\square}([\![(x,\mathord{\stackrel{n}{0}})]\!]_{n,\square}\otimes_{n,\square}[\![(\mathord{\stackrel{n}{0}},z)]\!]_{n,\square}).
\end{gather}
\end{subequations}
Similarly, one proves Eqs.~\eqref{eq:disco-3}, \eqref{eq:disco-6} and \eqref{eq:disco-7}.  The remaining cases follow from the fact that $(\mathbb{N}_{n,\square},\oplus_{n,\square},\otimes_{n,\square})$ is a commutative semiring and from the properties in Eq.~\eqref{eq:ties-argh}.  To prove that the commutative ring $(\mathbb{Z}_{n,\square},\oplus_{n,\square},\otimes_{n,\square})$ is an integral domain one has to prove that the multiplication $\otimes_{n,\square}$ is closed in $\mathbb{Z}_{n,\square}\setminus\{0_{n,\square}\}$.  If follows from the fact that the equations
\begin{equation}
F_{n}(F_{n+1}(x,u),F_{n+1}(y,v))=\mathord{\stackrel{n}{0}}\quad
\text{and}\quad
F_{n}(F_{n+1}(y,u),F_{n+1}(x,v))=\mathord{\stackrel{n}{0}},
\end{equation}
from Definition~\ref{def:biontech}, have no solutions in 
$\mathord{\stackrel{n}{\mathbb{N}}}$, except for $(x,y)$ or $(u,v)$ equal to $(0,0)$.
\end{proof}
\begin{remark}
From the properties in Eq.~\eqref{eq:ties-argh}, one obtains
\begin{align}
T_{n,\square}\:a
&=(T_{n,\square}\:1_{n,\square})\otimes_{n,\square}a.
\end{align}
generalizing the notion from $n=0$ that \textit{any number changes sign after multiplying by the negative unit.}
\end{remark}
\begin{remark}
The isomorphism $E_{\square}$ is nonexponential as defined in the last paragraph of section~\ref{sec:semirings}, since $\mathbb{Z}_{n,\square}\cap\mathbb{Z}_{n+1,\square}=\emptyset$.
\end{remark}

%% file: 7-fields.tex
\begin{notation}
Denote $\mathbb{Z}_{n,\square}^{\ast}=\mathbb{Z}_{n,\square}\setminus\{0_{n,\square}\}$.
\end{notation}
\begin{definition}\label{def:akira}
Define the direct product $(\mathbb{Z}_{n,\square}\times\mathbb{Z}_{n,\square}^{\ast},\otimes'_{n})$ such that
\begin{gather}
\forall(x,y),(u,v)\in\mathbb{Z}_{n,\square}\times\mathbb{Z}_{n,\square}^{\ast}:\nonumber\\
(x,y)\otimes'_{n}(u,v)=(x\otimes_{n,\square}u,y\otimes_{n,\square}v).
\end{gather}
\end{definition}
\begin{definition}\label{def:batman}
Define the function $\vec{E}_{\square}$ in $\bigcup_{n\in\mathbb{N}}\mathbb{Z}_{n,\square}\times\mathbb{Z}_{n,\square}^{\ast}$ as
\begin{equation}
\vec{E}_{\square}(x,y)=(E_{\square}(x),E_{\square}(y)).
\end{equation}
Denote by $\vec{L}_{\square}$ its inverse function.
\end{definition}
\begin{remark}
Note that $\mathbb{Z}_{n,\square}\times\mathbb{Z}_{n,\square}^{\ast}=\vec{E}_{\square}(\mathbb{Z}_{n-1,\square}\times\mathbb{Z}_{n-1,\square}^{\ast})$ and
\begin{equation}\label{eq:thomas}
(x,y)\otimes'_{n}(u,v)=\vec{E}_{\square}(\vec{L}_{\square}(x,y)\otimes'_{n-1}\vec{L}_{\square}(u,v)).
\end{equation}
\end{remark}
\begin{definition}
Define the binary operation
\begin{gather}
\oplus'_{n}:\mathbb{Z}_{n,\square}\times\mathbb{Z}_{n,\square}^{\ast}\rightarrow\mathbb{Z}_{n,\square}\times\mathbb{Z}_{n,\square}^{\ast}:\nonumber\\
(x,y)\oplus'_{n}(u,v)=((x\otimes_{n,\square}v)\oplus_{n,\square}(y\otimes_{n,\square}u),y\otimes_{n,\square}v).
\end{gather}
\end{definition}
\begin{proposition}\label{prop:river}
The tuple $(\mathbb{Z}_{n,\square}\times\mathbb{Z}_{n,\square}^{\ast},\otimes'_{n})$ is a commutative monoid.
\end{proposition}
\begin{proof}
It is a direct product of the commutative monoids $(\mathbb{Z}_{n,\square},\otimes_{n,\square})$ and $(\mathbb{Z}_{n,\square}^{\ast},\otimes_{n,\square})$.
\end{proof}
\begin{proposition}\label{prop:boca}
The tuple $(\mathbb{Z}_{n,\square}\times\mathbb{Z}_{n,\square}^{\ast},\oplus'_{n})$ is a commutative monoid.
\end{proposition}
\begin{proof}
Some expressions below are simplified using the standard order of operations where multiplication has precedence over addition. 
\begin{enumerate}
\item \textit{Closure.}  Let $(x,y),(u,v)\in\mathbb{Z}_{n,\square}\times\mathbb{Z}_{n,\square}^{\ast}$, be such that $(x,y)\oplus'_{n}(u,v)=(a,b)$.  From Theorem~\ref{th:taurus}, the tuple $(\mathbb{Z}_{n,\square},\oplus_{n,\square},\otimes_{n,\square})$ is an integral domain.  Since the latter is a ring and it has no nonzero zero divisors, then $a\in\mathbb{Z}_{n,\square}$ and $b\in\mathbb{Z}_{n,\square}^{\ast}$.
\item \textit{Identity element.} It is straightforward to verify that for all $(x,y)\in\mathbb{Z}_{n,\square}\times\mathbb{Z}_{n,\square}^{\ast}$, it holds that
\begin{equation}
(x,y)\oplus'_{n}(0_{n,\square},1_{n,\square})=(0_{n,\square},1_{n,\square})\oplus'_{n}(x,y)=(x,y).
\end{equation}
\item\textit{Commutativity.}
\begin{subequations}
\begin{align}
(x,y)\oplus'_{n}(u,v)&=((x\otimes_{n,\square}v)\oplus_{n,\square}(y\otimes_{n,\square}u),y\otimes_{n,\square}v)\\
&=((u\otimes_{n,\square}y)\oplus_{n,\square}(v\otimes_{n,\square}x),v\otimes_{n,\square}y)\\
&=(u,v)\oplus'_{n}(x,y).
\end{align}
\end{subequations}
\item \textit{Associativity.}  Let $(x_{i},y_{i})\in\mathbb{Z}_{n,\square}\times\mathbb{Z}_{n,\square}^{\ast}$ for $i=1,2,3$.
\begin{subequations}
\begin{gather}
(x_{1},y_{1})\oplus'_{n}((x_{2},y_{2})\oplus'_{n}(x_{3},y_{3}))\\
=(x_{1},y_{1})\oplus'_{n}(x_{2}\otimes_{n,\square}y_{3}\oplus_{n,\square}y_{2}\otimes_{n,\square}x_{3},y_{2}\otimes_{n,\square}y_{3})\\
=(x_{1}\otimes_{n,\square}y_{2}\otimes_{n,\square}y_{3}\oplus_{n,\square}y_{1}\otimes_{n,\square}(x_{2}\otimes_{n,\square}y_{3}\oplus_{n,\square}x_{3}\otimes_{n,\square}y_{2}),y_{1}\otimes_{n,\square}y_{2}\otimes_{n,\square}y_{3})\\
=((x_{1}\otimes_{n,\square}y_{2}\oplus_{n,\square}y_{1}\otimes_{n,\square}x_{2})\otimes_{n,\square}y_{3}\oplus_{n,\square}y_{1}\otimes_{n,\square}y_{2}\otimes_{n,\square}x_{3},y_{1}\otimes_{n,\square}y_{2}\otimes_{n,\square}y_{3})\\
=(x_{1}\otimes_{n,\square}y_{2}\oplus_{n,\square}y_{1}\otimes_{n,\square}x_{2},y_{1}\otimes_{n,\square}y_{2})\oplus'_{n}(x_{3},y_{3})\\
((x_{1},y_{1})\oplus'_{n}(x_{2},y_{2}))\oplus'_{n}(x_{3},y_{3}).
\end{gather}
\end{subequations}
\end{enumerate}
\end{proof}
\begin{proposition}\label{prop:gomez}
Let $(x,y),(u,v)\in\mathbb{Z}_{n,\square}\times\mathbb{Z}_{n,\square}^{\ast}$.  The cross-relation $\stackrel{n}{\boxtimes}$ defined as
\begin{equation}\label{eq:jorge}
(x,y)\stackrel{n}{\boxtimes}(u,v)\iff x\otimes_{n,\square}v=u\otimes_{n,\square}y,
\end{equation}
is a congruence relation on the tuple $(\mathbb{Z}_{n,\square}\times\mathbb{Z}_{n,\square}^{\ast},\oplus'_{n},\otimes'_{n})$.
\end{proposition}
\begin{proof}
Reflexivity and symmetry are straightforward.  It remains to prove transitivity and compatibility: 
\begin{enumerate}
\item \textit{Transitivity.} Let $(x_{i},y_{i})\in\mathbb{Z}_{n,\square}\times\mathbb{Z}_{n,\square}^{\ast}$ for $i=1,2,3$, such that
\begin{equation}
(x_{1},y_{1})\stackrel{n}{\boxtimes}(x_{2},y_{2})\quad
\text{and}\quad
(x_{2},y_{2})\stackrel{n}{\boxtimes}(x_{3},y_{3}),
\end{equation}
that is
\begin{equation}\label{eq:torres}
x_{1}\otimes_{n,\square}y_{2}=x_{2}\otimes_{n,\square}y_{1}\quad
\text{and}\quad
x_{2}\otimes_{n,\square}y_{3}=x_{3}\otimes_{n,\square}y_{2}.
\end{equation}
Multiplying both equations one obtains obtains
\begin{equation}
x_{1}\otimes_{n,\square}y_{2}\otimes_{n,\square}x_{2}\otimes_{n,\square}y_{3}=x_{2}\otimes_{n,\square}y_{1}\otimes_{n,\square}x_{3}\otimes_{n,\square}y_{2},
\end{equation}
that is
\begin{equation}
x_{2}\otimes_{n,\square}y_{2}\otimes_{n,\square}(x_{1}\otimes_{n,\square}y_{3}\oplus_{n,\square}T_{n,\square}\:(y_{1}\otimes_{n,\square}x_{3}))=0_{n,\square}.
\end{equation}
Since $(\mathbb{Z}_{n,\square},\oplus_{n,\square},\otimes_{n,\square})$ is an integral domain, either $x_{2}=0_{n,\square}$ or $x_{1}\otimes_{n,\square}y_{3}\oplus_{n,\square}T_{n,\square}\:(y_{1}\otimes_{n,\square}x_{3})=\mathord{\stackrel{n}{0}}$.  If $x_{2}=0_{n,\square}$, the Eq.~\eqref{eq:torres} implies $x_{1}=x_{3}=0_{n,\square}$, therefore $(x_{1},y_{1})\stackrel{n}{\boxtimes}(x_{3},y_{3})$.  If $x_{1}\otimes_{n,\square}y_{3}\oplus_{n,\square}T_{n,\square}\:(y_{1}\otimes_{n,\square}x_{3})=0_{n,\square}$, then $x_{1}\otimes_{n,\square}y_{3}=x_{3}\otimes_{n,\square}y_{1}$, that is $(x_{1},y_{1})\stackrel{n}{\boxtimes}(x_{3},y_{3})$.

\item \textit{Compatibility.}  Let $(x_{i},y_{i}),(u_{i},v_{i})\in\mathbb{Z}_{n,\square}\times\mathbb{Z}_{n,\square}^{\ast}$ for $i=1,2$, such that
\begin{equation}\label{eq:cadillac}
(x_{1},y_{1})\stackrel{n}{\boxtimes}(x_{2},y_{2})\quad
\text{and}\quad
(u_{1},v_{1})\stackrel{n}{\boxtimes}(u_{2},v_{2}).
\end{equation}
For $i=1,2$, denote
\begin{subequations}
\begin{align}
(a_{i},b_{i})&=(x_{i},y_{i})\oplus'_{n}(u_{i},v_{i})=(x_{i}\otimes_{n,\square}v_{i}\oplus_{n,\square}y_{i}\otimes_{n,\square}u_{i},y_{i}\otimes_{n,\square}v_{i}),\\
(a'_{i},b'_{i})&=(x_{i},y_{i})\otimes'_{n}(u_{i},v_{i})=(x_{i}\otimes_{n,\square}u_{i},y_{i}\otimes_{n,\square}v_{i}).
\end{align}
\end{subequations}
From Eq.~\eqref{eq:cadillac} it follows that
\begin{equation}
a_{1}\otimes_{n,\square}b_{2}=a_{2}\otimes_{n,\square}b_{1}\quad
\text{and}\quad
a'_{1}\otimes_{n,\square}b'_{2}=a'_{2}\otimes_{n,\square}b'_{1},
\end{equation}
implying
\begin{equation}
(a_{1},b_{1})\stackrel{n}{\boxtimes}(a_{2},b_{2})\quad
\text{and}\quad
(a'_{1},b'_{1})\stackrel{n}{\boxtimes}(a'_{2},b'_{2}).
\end{equation}
\end{enumerate}
\end{proof}
\begin{notation}
Consider the quotient structure
\begin{equation}
(\mathbb{Q}_{n,\boxtimes},\oplus_{n,\boxtimes},\otimes_{n,\boxtimes})=\left(\frac{\mathbb{Z}_{n,\square}\times\mathbb{Z}_{n,\square}^{\ast}}{\stackrel{n}{\boxtimes}},\oplus_{n,\boxtimes},\otimes_{n,\boxtimes}\right),
\end{equation}
where $\oplus_{n,\boxtimes}$ and $\otimes_{n,\boxtimes}$ are the operations induced on $(\mathbb{Z}_{n,\square}\times\mathbb{Z}_{n,\square}^{\ast})/\!\stackrel{n}{\boxtimes}$ by $\oplus'_{n}$ and $\otimes'_{n}$.  Denote by $[\![(x,y)]\!]_{n,\boxtimes}$ the equivalence class of $(x,y)$ under $\stackrel{n}{\boxtimes}$.
\end{notation}
\begin{remark}
If $(x,y)\stackrel{n}{\boxtimes}(u,v)$ then $(T_{n,\square}\:x,y)\stackrel{n}{\boxtimes}(T_{n,\square}\:u,v)$ and $(y,x)\stackrel{n}{\boxtimes}(v,u)$.
\end{remark}
\begin{notation}
Denote $k_{n,\boxtimes}=[\![(k_{n,\square},1_{n,\square})]\!]_{n,\boxtimes}$ and $\mathbb{Q}_{n,\boxtimes}^{\ast}=\mathbb{Q}_{n,\boxtimes}\setminus\{0_{n,\boxtimes}\}$.
\end{notation}
\begin{definition}\label{def:kaliman}
Define the involution
\begin{gather}
T_{0,\boxtimes}\: :\mathbb{Q}_{0,\boxtimes}\rightarrow\mathbb{Q}_{0,\boxtimes}:\nonumber\\
T_{0,\boxtimes}\:[\![(x,y)]\!]_{0,\boxtimes}=[\![(T_{0,\square}\:x,y)]\!]_{0,\boxtimes}=[\![(x,T_{0,\square}\:y)]\!]_{0,\boxtimes}.
\end{gather}
And for $n>0$, define the involution
\begin{subequations}
\begin{gather}
T_{n,\boxtimes}\: :\mathbb{Q}_{n-1,\boxtimes}^{\ast}\cup\mathbb{Q}_{n,\boxtimes}\rightarrow\mathbb{Q}_{n-1,\boxtimes}^{\ast}\cup\mathbb{Q}_{n,\boxtimes}:\nonumber\\
T_{n,\boxtimes}\:[\![(x,y)]\!]_{n-1,\boxtimes}=[\![(y,x)]\!]_{n-1,\boxtimes},\\
T_{n,\boxtimes}\:[\![(x,y)]\!]_{n,\boxtimes}=[\![(T_{n,\square}\:x,y)]\!]_{n,\boxtimes}=[\![(x,T_{n,\square}\:y)]\!]_{n,\boxtimes}.
\end{gather}
\end{subequations}
\end{definition}
\begin{theorem}
The tuple $(\mathbb{Q}_{n,\boxtimes},\oplus_{n,\boxtimes},\otimes_{n,\boxtimes})$ is a commutative field.
\end{theorem}
\begin{proof}
From Propositions~\ref{prop:river},~\ref{prop:boca} and~\ref{prop:gomez}, the tuples $(\mathbb{Q}_{n,\boxtimes},\oplus_{n,\boxtimes})$ and $(\mathbb{Q}_{n,\boxtimes},\otimes_{n,\boxtimes})$ are commutative monoids induced by $(\mathbb{Z}_{n,\square}\times\mathbb{Z}_{n,\square}^{\ast},\oplus'_{n})$ and $(\mathbb{Z}_{n,\square}\times\mathbb{Z}_{n,\square}^{\ast},\otimes'_{n})$.  The following properties rely on the equivalence class $[\![.]\!]_{n,\boxtimes}$ and complete the proof that the tuple $(\mathbb{Q}_{n,\boxtimes},\oplus_{n,\boxtimes},\otimes_{n,\boxtimes})$ is a commutative field:
\begin{enumerate}
\item \textit{Additive inverse.} Let $(x,y)\in\mathbb{Z}_{n,\square}\times\mathbb{Z}_{n,\square}^{\ast}$, then
\begin{subequations}
\begin{align}
[\![(x,y)]\!]_{n,\boxtimes}
\oplus_{n,\boxtimes}T_{n,\square}\:[\![(x,y)]\!]_{n,\boxtimes}
&=[\![(x,y)]\!]_{n,\boxtimes}
\oplus_{n,\boxtimes}[\![(T_{n,\square}\:x,y)]\!]_{n,\boxtimes}\\
&=[\![(0_{n,\square},y\otimes_{n,\square}y)]\!]_{n,\boxtimes}\\
&=[\![(0_{n,\square},1_{n,\square})]\!]_{n,\boxtimes}.
\end{align}
\end{subequations}
\item \textit{Mutiplicative inverse.}
\begin{subequations}
\begin{align}
[\![(x,y)]\!]_{n,\boxtimes}
\otimes_{n,\boxtimes}T_{n+1,\square}\:[\![(x,y)]\!]_{n,\boxtimes}
&=[\![(x,y)]\!]_{n,\boxtimes}
\otimes_{n,\boxtimes}[\![(y,x)]\!]_{n,\boxtimes}\\
&=[\![(x\otimes_{n,\square}y,y\otimes_{n,\square}x)]\!]_{n,\boxtimes}\\
&=[\![(1_{n,\square},1_{n,\square})]\!]_{n,\boxtimes}.
\end{align}
\end{subequations}
\item \textit{Distributivity.}  Let $(x_{i},y_{i})\in\mathbb{Z}_{n,\square}\times\mathbb{Z}_{n,\square}^{\ast}$ for $i=1,2,3$, then
\begin{subequations}
\begin{gather}
[\![(x_{1},y_{1})]\!]_{n,\boxtimes}\otimes_{n,\boxtimes}([\![(x_{2},y_{2})]\!]_{n,\boxtimes}\oplus_{n,\boxtimes}[\![(x_{3},y_{3})]\!]_{n,\boxtimes})\\
[\![(x_{1},y_{1})]\!]_{n,\boxtimes}\otimes_{n,\boxtimes}[\![(x_{2}\otimes_{n,\square}y_{3}\oplus_{n,\square}y_{2}\otimes_{n,\square}x_{3},y_{2}\otimes_{n,\square}y_{3})]\!]_{n,\boxtimes}\\
[\![(x_{1}\otimes_{n,\square}(x_{2}\otimes_{n,\square}y_{3}\oplus_{n,\square}y_{2}\otimes_{n,\square}x_{3}),y_{1}\otimes_{n,\square}y_{2}\otimes_{n,\square}y_{3})]\!]_{n,\boxtimes}\\
[\![(x_{1}\otimes_{n,\square}x_{2}\otimes_{n,\square}y_{3}\oplus_{n,\square}x_{1}\otimes_{n,\square}y_{2}\otimes_{n,\square}x_{3},y_{1}\otimes_{n,\square}y_{2}\otimes_{n,\square}y_{3})]\!]_{n,\boxtimes}\\
[\![(x_{1}\otimes_{n,\square}x_{2}\otimes_{n,\square}y_{1}\otimes_{n,\square}y_{3}\oplus_{n,\square}y_{1}\otimes_{n,\square}y_{2}\otimes_{n,\square}x_{1}\otimes_{n,\square}x_{3},y_{1}\otimes_{n,\square}y_{2}\otimes_{n,\square}y_{1}\otimes_{n,\square}y_{3})]\!]_{n,\boxtimes}\\
[\![(x_{1}\otimes_{n,\square}x_{2},y_{1}\otimes_{x,\square}y_{2})]\!]_{n,\boxtimes}\oplus_{n,\boxtimes}[\![(x_{1}\otimes_{n,\square}x_{3},y_{1}\otimes_{n,\square}y_{3})]\!]_{n,\boxtimes}\\
[\![(x_{1},y_{1})]\!]_{n,\boxtimes}\otimes_{n,\boxtimes}[\![(x_{2},y_{2})]\!]_{n,\boxtimes}\oplus_{n,\boxtimes}[\![(x_{1},y_{1})]\!]_{n,\boxtimes}\otimes_{n,\boxtimes}[\![(x_{3},y_{3})]\!]_{n,\boxtimes}.
\end{gather}
\end{subequations}
\end{enumerate}
\end{proof}
\begin{remark}
Denote $\mathbb{Z}^{+}_{n,\square}=\mathbb{N}_{n,\square}\setminus \{0_{n,\square}\}$.  Any element $[\![(x,y)]\!]_{n,\boxtimes}$ in $\mathbb{Q}_{n,\boxtimes}$ can be rewritten as $[\![(x',y')]\!]_{n,\boxtimes}$, where $y'\in\mathbb{Z}^{+}_{n,\square}$.
\end{remark}
\begin{definition}\label{def:neverland}
The ordering $\stackrel{n}{<}$ in $\mathbb{Q}_{n,\boxtimes}$ is defined as
\begin{equation}
[\![(x,y)]\!]_{n,\boxtimes}\stackrel{n}{<}[\![(u,v)]\!]_{n,\boxtimes}\iff (x'\otimes_{n,\square}v')\stackrel{n}{<}(y'\otimes_{n,\square}u'),
\end{equation}
where $[\![(x,y)]\!]_{n,\boxtimes}=[\![(x',y')]\!]_{n,\boxtimes}$ and $[\![(u,v)]\!]_{n,\boxtimes}=[\![(u',v')]\!]_{n,\boxtimes}$, with $y',v'\in\mathbb{Z}^{+}_{n,\square}$.
\end{definition}
\begin{proposition}\label{prop:odesa}
Let $(x,y),(u,v)\in\mathbb{Z}_{n,\square}\times\mathbb{Z}_{n,\square}^{\ast}$, then
\begin{equation}
(x,y)\stackrel{n}{\boxtimes}(u,v)\iff \vec{E}_{\square}(x,y)\stackrel{n+1}{\boxtimes}\vec{E}_{\square}(u,v).
\end{equation}
\end{proposition}
\begin{proof} 
The proof that $(x,y)\stackrel{n}{\boxtimes}(u,v)\Rightarrow\vec{E}_{\square}(x,y)\stackrel{n+1}{\boxtimes}\vec{E}_{\square}(u,v)$ is the following: If $(x,y)\stackrel{n}{\boxtimes}(u,v)$ then $F_{n+1}(x,v)=F_{n+1}(y,u)$.  Applying the exponential $E$ at both sides, one obtains $F_{n+2}(E(x),E(v))=F_{n+2}(E(y),E(u))$.  From the definition of $\stackrel{n+1}{\boxtimes}$, this implies $(E(x),E(y))\stackrel{n+1}{\boxtimes}(E(u),E(v))$, that is, $\vec{E}_{\square}(x,y)\stackrel{n+1}{\boxtimes}\vec{E}_{\square}(u,v)$.  Noting that $E$ is an injective function, the converse statement is proven in a similar way. 
\end{proof}
\begin{definition}\label{def:trump}
Define the function $E_{\boxtimes}$ in $U_{\mathbb{Q}}=\bigcup_{n\in\mathbb{N}}\mathbb{Q}_{n,\boxtimes}$ as
\begin{equation}
E_{\boxtimes}([\![a]\!]_{n,\boxtimes})=[\![\vec{E}_{\square}(a)]\!]_{n+1,\boxtimes}.
\end{equation}
Denote by $L_{\boxtimes}$ its inverse function.
\end{definition}
\begin{remark}
Note that $\mathbb{Q}_{n,\boxtimes}=E_{\boxtimes}(\mathbb{Q}_{n-1,\boxtimes})$.  
\end{remark}
\begin{remark}
The isomorphism $E_{\boxtimes}$ is nonexponential as defined in the last paragraph of section~\ref{sec:semirings}, since $\mathbb{Q}_{n,\boxtimes}\cap\mathbb{Q}_{n+1,\boxtimes}=\emptyset$.
\end{remark}
\begin{proposition}
Let $a,b\in\mathbb{Q}_{n,\boxtimes}$, then
\begin{subequations}
\begin{align}
a\otimes_{n,\boxtimes}b=E_{\boxtimes}(L_{\boxtimes}(a)\otimes_{n-1,\boxtimes}L_{\boxtimes}(b)),\\
a\oplus_{n,\boxtimes}b=E_{\boxtimes}(L_{\boxtimes}(a)\oplus_{n-1,\boxtimes}L_{\boxtimes}(b)).
\end{align}
\end{subequations}
\end{proposition}
\begin{proof}
Let $a=[\![(x,y)]\!]_{n,\boxtimes}$ and $b=[\![(u,v)]\!]_{n,\boxtimes}$. From Propositions \ref{prop:piscis} and \ref{prop:aries}, one obtains
\begin{subequations}
\begin{gather}
[\![(x,y)]\!]_{n,\boxtimes}\oplus_{n,\boxtimes}[\![(u,v)]\!]_{n,\boxtimes}\\
=[\![((x\otimes_{n,\square}v)\oplus_{n,\square}(y\otimes_{n,\square}u),y\otimes_{n,\square}v)]\!]_{n,\boxtimes}\\
=[\![(E_{\square}(L_{\square}(x)\otimes_{n-1,\square}L_{\square}(v))\oplus_{n-1,\square}(L_{\square}(y)\otimes_{n-1,\square}L_{\square}(u)),E_{\square}(L_{\square}(y)\otimes_{n-1,\square}L_{\square}(v)))]\!]_{n,\boxtimes}\\
=E_{\boxtimes}([\![(L_{\square}(x),L_{\square}(y))]\!]_{n-1,\square}\oplus_{n-1,\boxtimes}[\![(L_{\square}(u),L_{\square}(v))]\!]_{n-1,\square})\\
=E_{\boxtimes}(L_{\boxtimes}([\![(x,y)]\!]_{n,\square})\oplus_{n-1,\boxtimes}L_{\boxtimes}([\![(u,v)]\!]_{n,\square}))
\end{gather}
\end{subequations}
and
\begin{subequations}
\begin{gather}
[\![(x,y)]\!]_{n,\boxtimes}\otimes_{n,\boxtimes}[\![(u,v)]\!]_{n,\boxtimes}\\
=[\![(x\otimes_{n,\square}u,y\otimes_{n,\square}v)]\!]_{n,\boxtimes}\\
=[\![(E_{\square}(L_{\square}(x)\otimes_{n-1,\square}L_{\square}(u)),E_{\square}(L_{\square}(y)\otimes_{n-1,\square}L_{\square}(v)))]\!]_{n-1,\boxtimes}\\
=E_{\boxtimes}([\![(L_{\square}(x),L_{\square}(y))]\!]_{n-1,\boxtimes}\otimes_{n-1,\boxtimes}[\![(L_{\square}(u),L_{\square}(v))]\!]_{n-1,\boxtimes})\\
=E_{\boxtimes}(L_{\boxtimes}([\![(x,y)]\!]_{n,\boxtimes})\otimes_{n-1,\boxtimes}L_{\boxtimes}([\![(u,v)]\!]_{n,\boxtimes})).
\end{gather}
\end{subequations}
\end{proof}
\begin{notation}
 Denote $\mathbb{N}_{n,\boxtimes}=[\![(\mathbb{N}_{n,\square},1_{n,\square})]\!]_{n,\boxtimes}$, $\mathbb{Z}_{n,\boxtimes}=[\![(\mathbb{Z}_{n,\square},1_{n,\square})]\!]_{n,\boxtimes}$ and $\mathbb{Z}_{n,\boxtimes}^{\ast}=[\![(\mathbb{Z}_{n,\square}^{\ast},1_{n,\square})]\!]_{n,\boxtimes}$.
\end{notation}
\begin{remark}
Every element in the commutative field $(\mathbb{Q}_{n,\boxtimes},\oplus_{n,\boxtimes},\otimes_{n,\boxtimes})$ can be written as $a\otimes_{n,\boxtimes}(T_{n+1,\boxtimes}\:b)$, where $a\in\mathbb{Z}_{n,\boxtimes}$ and $b\in\mathbb{Z}_{n,\boxtimes}^{\ast}$.
\end{remark}

%% file: 8-fields.tex
There are several ways to construct the usual real numbers ($\mathbb{R}_{0,\boxtimes}$) from the usual rational numbers ($\mathbb{Q}_{0,\boxtimes}$).  Analogously, one can construct the reals $\mathbb{R}_{n,\boxtimes}$ from $\mathbb{Q}_{n,\boxtimes}$. Define the absolute value in $\mathbb{Q}_{n,\boxtimes}$ as
\begin{equation}\label{eq:absval}
|a|_{n,\boxtimes}
=\left\{\begin{array}{ll}T_{n,\boxtimes}\:a,&\text{if}\ a\stackrel{n}{<}0_{n,\boxtimes};\\
0_{n,\boxtimes},&\text{if}\ a=0_{n,\boxtimes};\\
a,&\text{if}\ a\stackrel{n}{>}0_{n,\boxtimes}.
\end{array}\right.
\end{equation}
In the method of Cauchy sequences, the set $\mathbb{R}_{n,\boxtimes}$ is the Cauchy completion of $\mathbb{Q}_{n,\boxtimes}$ with respect to the metric 
\begin{equation}
d_{n,\boxtimes}(a,b)=|a\oplus_{n,\boxtimes}T_{n,\boxtimes}\:b|_{n,\boxtimes}.
\end{equation}
Any number in $\mathbb{R}_{n,\boxtimes}$ is associated to the equivalence class of a Cauchy sequence whose elements belongs to $\mathbb{Q}_{n,\boxtimes}$.  The equivalence class, denoted $[\![\cdot]\!]_{n,\mathbb{R}}$, is that of Cauchy sequences whose distance in the metric $d_{n,\boxtimes}$ tends to the additive identity $0_{n,\boxtimes}$.  That is, if $a=(a_{0},a_{1},a_{2},\dotsc)$ and $b=(b_{0},b_{1},b_{2},\dotsc)$ are two Cauchy sequences in $\mathbb{Q}_{n,\boxtimes}$, then
\begin{equation}
[\![a]\!]_{n,\mathbb{R}}=[\![b]\!]_{n,\mathbb{R}}\iff\lim_{k\rightarrow\infty}d_{n,\boxtimes}(a_{k},b_{k})=0_{n,\boxtimes}.
\end{equation}
  Any function in $\mathbb{Q}_{n,\boxtimes}$ is naturally extended to $\mathbb{R}_{n,\boxtimes}$ by applying the former to every element of the representative Cauchy sequence.

%% file: 9-embedding.tex
 The three characterisitc constructions previously studied: $(\mathord{\stackrel{n}{\mathbb{N}}},F_{n},F_{n+1})$, $(\mathbb{Z}_{n,\square},\oplus_{n,\square},\otimes_{n,\square})$ and $(\mathbb{Q}_{n,\square},\oplus_{n,\boxtimes},\otimes_{n,\boxtimes})$, are embedded in $\mathbb{R}_{0,\boxtimes}$ using Euler's real exponential function.  The resulting structures are denoted $(\mathbb{N}_{n},\bullet_{n},\bullet_{n+1})$, $(\mathbb{Z}_{n},\bullet_{n},\bullet_{n+1})$ and $(\mathbb{Q}_{n},\bullet_{n},\bullet_{n+1})$.  The section ends with remarks pointing to future research.
\begin{notation}
Denote $\exp_{\omega}(a)=\exp(a\ln(\omega))$, where $\exp(.)$ is Euler's real exponential function in $\mathbb{R}_{0,\boxtimes}$, $\ln(.)$ its inverse function and $\omega$ is the image of $w_{0}\in\mathord{\stackrel{0}{\mathbb{N}}}$ in $\mathbb{R}_{0,\boxtimes}$.
Denote by $\ln_{\omega}(.)$ the inverse function of $\exp_{\omega}(.)$.  Denote
\begin{equation}
\mathbb{R}_{n}=\exp_{\omega}^{n}(\mathbb{R}_{0}),\quad
\mathbb{Q}_{n}=\exp_{\omega}^{n}(\mathbb{Q}_{0}),\quad
\mathbb{Z}_{n}=\exp_{\omega}^{n}(\mathbb{Z}_{0})\quad
\text{and}\quad
\mathbb{N}_{n}=\exp_{\omega}^{n}(\mathbb{N}_{0}),
\end{equation}
where $\mathbb{R}_{0}=\mathbb{R}_{0,\boxtimes}$, $\mathbb{Q}_{0}=\mathbb{Q}_{0,\boxtimes}$, $\mathbb{Z}_{0}=\mathbb{Z}_{0,\boxtimes}$ and $\mathbb{N}_{0}=\mathbb{N}_{0,\boxtimes}$.
\end{notation}
\begin{remark}
Unlike the sets $\mathbb{Z}_{n}$ and $\mathbb{Q}_{n}$, the sets $\mathbb{N}_{n}$ and $\mathbb{R}_{n}$ share the property that $\mathbb{N}_{n}\subset\mathbb{N}_{n+1}$ and $\mathbb{R}_{n}\subset\mathbb{R}_{n+1}$.  So far it was assumed that $\omega\in\mathbb{N}_{0}\setminus\{0_{0},1_{0}\}$.  Note that if $\omega$ becomes a positive, noninteger real number, the property $\mathbb{N}_{n}\subset\mathbb{N}_{n+1}$ is no longer valid.
\end{remark}
\begin{remark}
Denote $\bullet_{0}$ and $\bullet_{1}$ the addition and multiplication in $\mathbb{R}_{0}$.  That is, if $a=(a_{0},a_{1},a_{2},\dotsc)$ and $b=(b_{0},b_{1},b_{2},\dotsc)$ are two Cauchy sequences in $\mathbb{Q}_{0}$, then $[\![a]\!]_{0,\mathbb{R}},[\![b]\!]_{0,\mathbb{R}}\in\mathbb{R}_{0}$ and
\begin{subequations}
\begin{align}
[\![a]\!]_{0,\mathbb{R}}\bullet_{0}[\![b]\!]_{0,\mathbb{R}}&=[\![(a_{0}\oplus_{0,\boxtimes}b_{0},a_{1}\oplus_{0,\boxtimes}b_{1},a_{2}\oplus_{0,\boxtimes}b_{2},\dotsc)]\!]_{0,\mathbb{R}},\\
[\![a]\!]_{0,\mathbb{R}}\bullet_{1}[\![b]\!]_{0,\mathbb{R}}&=[\![(a_{0}\otimes_{0,\boxtimes}b_{0},a_{1}\otimes_{0,\boxtimes}b_{1},a_{2}\otimes_{0,\boxtimes}b_{2},\dotsc)]\!]_{0,\mathbb{R}}.
\end{align}
\end{subequations}
\end{remark}
\begin{definition}\label{def:siberia}
For $n\geq 2$, define the binary operations
\begin{gather}
\bullet_{n}:(\mathbb{R}_{n-1})^{2}\rightarrow\mathbb{R}_{n-1}:\nonumber\\
a\bullet_{n}b=\exp_{\omega}(\ln_{\omega}(a)\bullet_{n-1}\ln_{\omega}(b)).
\end{gather}
%
\end{definition} 
\begin{notation}
Denote $\mathbb{N}_{n}=\{0_{n},1_{n},2_{n},\dotsc\}$. 
\end{notation}
\begin{notation}
Denote by $T_{n}$ and $T_{n+1}$, the additive and multiplicative inverse in $(\mathbb{R}_{n},\bullet_{n},\bullet_{n+1})$.  Note that for all $n\in\mathbb{N}$, $T_{n+1}$ is defined in $\mathbb{R}_{n}\setminus\{0_{n}\}$ and $T_{0}$ is defined in $\mathbb{R}_{0}$.
\end{notation}
\begin{notation}
Denote
\begin{align}
U_{\mathbb{Z},\square}=\bigcup_{n\in\mathbb{N}}\mathbb{Z}_{n,\square},\quad
U_{\mathbb{Q},\boxtimes}=\bigcup_{n\in\mathbb{N}}\mathbb{Q}_{n,\boxtimes},\quad
U_{\mathbb{Z}}=\bigcup_{n\in\mathbb{N}}\mathbb{Z}_{n},\quad
U_{\mathbb{Q}}=\bigcup_{n\in\mathbb{N}}\mathbb{Q}_{n}.
\end{align}
\end{notation}
\begin{remark}
The isomorphism $\exp_{\omega}(.)$ is exponential as defined in the last paragraph of Section~\ref{sec:semirings}, since $\mathbb{R}_{n}\cap\mathbb{R}_{n+1}$ is a nonempty set where
\begin{equation}\label{eq:lebron}
{\rm Ad}_{\exp^{n+1}_{\omega}}(\bullet_{0})={\rm Ad}_{\exp^{n}_{\omega}}(\bullet_{1}).
\end{equation}
Similarly, the restrictions of $\exp_{\omega}(.)$ to $U_{\mathbb{Z}}$ and $U_{\mathbb{Q}}$ are exponential, since $\mathbb{Z}_{n}\cap\mathbb{Z}_{n+1}$ and $\mathbb{Q}_{n}\cap\mathbb{Q}_{n+1}$ are nonempty sets fulfilling Eq.~\eqref{eq:lebron}.  
\end{remark}
\begin{definition}
Define the maps
\begin{subequations}
\begin{gather}
\psi:\mathord{\stackrel{0}{\mathbb{N}}}\rightarrow\mathbb{N}_{0}:\nonumber\\
\psi(k)=[\![(k_{0,\boxtimes},k_{0,\boxtimes},k_{0,\boxtimes},\dotsc)]\!]_{0,\boxtimes},\\
\psi_{\square}:U_{\mathbb{Z},\square}\rightarrow U_{\mathbb{Z}}:\nonumber\\
\psi_{\square}([\![(x,y)]\!]_{n,\square})
=\psi(x)\bullet_{n}T_{n}\:\psi(y),\\
\psi_{\boxtimes}:U_{\mathbb{Q},\boxtimes}\rightarrow U_{\mathbb{Q}}:\nonumber\\
\psi_{\boxtimes}([\![(x,y)]\!]_{n,\boxtimes})
=\psi_{\square}(x)\bullet_{n+1}T_{n+1}\:\psi_{\square}(y).
\end{gather}
\end{subequations}
\end{definition}
\begin{remark}
The restriction $\psi\mathord{\restriction}_{\mathord{\stackrel{n}{\mathbb{N}}}}$ is a semiring isomorphism from $(\mathord{\stackrel{n}{\mathbb{N}}},F_{n},F_{n+1})$ to $(\mathbb{N}_{n},\bullet_{n},\bullet_{n+1})$, with $\omega=\psi(w_{0})$.  The restriction $\psi_{\square}\mathord{\restriction}_{\mathbb{Z}_{n,\square}}$
is an integral domain isomorphism from $(\mathbb{Z}_{n,\square},\oplus_{n,\square},\otimes_{n,\square})$ to $(\mathbb{Z}_{n},\bullet_{n},\bullet_{n+1})$.  The restriction $\psi_{\boxtimes}\mathord{\restriction}_{\mathbb{Q}_{n,\boxtimes}}$ is a field isomorphism from $(\mathbb{Q}_{n,\boxtimes},\oplus_{n,\boxtimes},\otimes_{n,\boxtimes})$ to $(\mathbb{Q}_{n},\bullet_{n},\bullet_{n+1})$.  Some equations of general interest:
\begin{itemize}
\item Let $a\in\mathord{\stackrel{n}{\mathbb{N}}}$, then
\begin{equation}
\exp_{\omega}(\psi(a))=\psi(E(a)).
\end{equation}
\item Let $a\in\mathbb{Z}_{n,\square}$, then
\begin{subequations}
\begin{align}
\exp_{\omega}(\psi_{\square}(a))&=\psi_{\square}(E_{\square}(a)),\\
T_{n}\:(\psi_{\square}(a))&=\psi_{\square}(T_{n,\square}\:a).
\end{align}
\end{subequations}
\item Let $a\in\mathbb{Q}_{n,\boxtimes}$, then
\begin{subequations}
\begin{align}
\exp_{\omega}(\psi_{\boxtimes}(a))&=\psi_{\boxtimes}(E_{\boxtimes}(a)),\\
T_{n}\:(\psi_{\boxtimes}(a))&=\psi_{\boxtimes}(T_{n,\boxtimes}\:a),\\
T_{n+1}\:(\psi_{\boxtimes}(a))&=\psi_{\boxtimes}(T_{n+1,\boxtimes}\:a).
\end{align}
\end{subequations}
\end{itemize}
\end{remark}
\begin{remark}
The functions $\psi_{\square}$ and $\psi_{\boxtimes}$ can be interpreted as equivalence classes in $U_{\mathbb{Z},\square}$ and $U_{\mathbb{Q},\boxtimes}$.  Some examples:
\begin{itemize}
\item For $w_{0}=\mathord{\stackrel{0}{2}}$, the numbers $16_{0,\square},4_{1,\square},2_{2,\square}\in U_{\mathbb{Z},\square}$ belong to the same class since 
\begin{equation}
\psi_{\square}(16_{0,\square})=\psi_{\square}(4_{1,\square})=\psi_{\square}(2_{2,\square}).
\end{equation}
\item For $w_{0}=\mathord{\stackrel{0}{4}}$, the numbers $2_{0,\boxtimes},T_{2,\boxtimes}\:2_{1,\boxtimes},T_{3,\boxtimes}\:T_{2,\boxtimes}\:2_{2,\boxtimes}\in U_{\mathbb{Q},\boxtimes}$ belong to the same class since 
\begin{equation}
\psi_{\boxtimes}(2_{0,\boxtimes})=\psi_{\boxtimes}(T_{2,\boxtimes}\:2_{1,\boxtimes})=\psi_{\boxtimes}(T_{3,\boxtimes}\:T_{2,\boxtimes}\:2_{2,\boxtimes}).
\end{equation}
\end{itemize}
A question left to the reader is whether it can find examples of elements $a\in\mathbb{Q}_{n,\boxtimes}$ and $b\in\mathbb{Q}_{m,\boxtimes}$ such that $\psi_{\boxtimes}(a)=\psi_{\boxtimes}(b)$, without any $c\in\mathbb{Q}_{r,\boxtimes}$ such that $\psi_{\boxtimes}(a)=\psi_{\boxtimes}(b)=\psi_{\boxtimes}(c)$, where $n<r<m$.
\end{remark}
\begin{remark}
A \textit{groupoid category} is a category in which every morphism is invertible.
This concept provides a general language where to study the former sequences of algebraic sets.  For example, the groupoid category with objects $\{(\mathbb{Q}_{n,\boxtimes},\oplus_{n,\boxtimes},\otimes_{n,\boxtimes})\}_{n\in\mathbb{N}}$ and morphisms
\begin{equation}
E_{\boxtimes}^{m}:\mathbb{Q}_{n,\boxtimes}\rightarrow\mathbb{Q}_{n+m,\boxtimes}\quad\text{and}\quad L_{\boxtimes}^{m}:\mathbb{Q}_{n+m,\boxtimes}\rightarrow\mathbb{Q}_{n,\boxtimes},
\end{equation}
for $n,m\in\mathbb{N}$, is isomorphic to the groupoid category with objects $\{(\mathbb{Q}_{n},\bullet_{n},\bullet_{n+1})\}_{n\in\mathbb{N}}$ and morphisms  
\begin{equation}
\exp_{\omega}^{m}:\mathbb{Q}_{n}\rightarrow\mathbb{Q}_{n+m}\quad\text{and}\quad \log_{\omega}^{m}:\mathbb{Q}_{n+m}\rightarrow\mathbb{Q}_{n}.
\end{equation}
The functor between their objects is a set of field isomorphisms.
\end{remark}
\begin{remark} 
Denote $\mathbb{R}_{n}^{+}=\{x\in\mathbb{R}_{n}:x>0_{n}\}$.  The sequence of fields $\{(\mathbb{R}_{n},\bullet_{n},\bullet_{n+1})\}_{n\in\mathbb{N}}$ can be realized with the relaxed condition $\omega\in\mathbb{R}_{0}^{+}\setminus\{1_{0}\}$, opening new possibilities.  The so called \textit{infinite power tower} is the function
\begin{equation}\label{eq:inftower}
h(x)=\lim_{n\rightarrow\infty}\exp_{x}^{n}(x),
\end{equation}  
converging within $\mathbb{R}_{0}$ in the real interval $\exp(T_{0}\:e)<x<\exp(T_{1}\:e)$, where $e$ is Euler's number. 
At $x=\omega$, one obtains $h(\omega)=\lim_{n\rightarrow\infty}0_{n}$, a fixed point of the isomorphisms between the fields $\{(\mathbb{R}_{n},\bullet_{n},\bullet_{n+1})\}_{n\in\mathbb{N}}$.
\end{remark}
\begin{remark}
Noting that $\mathbb{R}_{n+1}=\exp_{\omega}(\mathbb{R}_{n})$, it is easy to see that $\mathbb{R}_{n+1}=\mathbb{R}_{n}^{+}$.  
\end{remark}
\begin{remark}
The tuple
$(\mathbb{R}_{n+1}\cup\{0_{n}\},\bullet_{n},\bullet_{n+1})$ is a semiring. The case $n=0$ is the semiring of nonnegative reals, known as the \textit{probability semiring} in the context of weighted automata \cite{SA09}.  
The case $n=-1$ corresponds to the \textit{log semiring} \cite{DR09}, where $0_{-1}=\log_{\omega}(0_{0})$ and 
the operation $\bullet_{-1}$ is defined as
\begin{equation}\label{eq:harris}
a\bullet_{-1}b=\log_{\omega}(\exp_{\omega}(a)\bullet_{0}\exp_{\omega}(b))
\end{equation}
In the limit $\omega\rightarrow+\infty$, it becomes the \textit{max tropical semiring} where $a\bullet_{-1}b=\max\{a,b\}$.
\end{remark}
\begin{remark}
Denote the extended natural numbers as $\mathord{\stackrel{0}{\mathbb{N}}}\cup\{\infty\}$, and extend the domain of the binary operation $F_{0}$, such that $F_{0}(a,\infty)=\infty$, for all $a\in\mathord{\stackrel{0}{\mathbb{N}}}\cup\{\infty\}$.  Construct --via inverse completion-- the extended integer numbers as $\mathbb{Z}_{0,\square}\cup\{[\![(\mathord{\stackrel{0}{0}},\infty)]\!]_{0,\square},[\![(\infty,\mathord{\stackrel{0}{0}})]\!]_{0,\square}\}$, extending the domain of the binary operation $\oplus_{n,\square}$, such that
\begin{enumerate}
\item If $a\in\mathbb{Z}_{0,\square}\cup\{[\![(\mathord{\stackrel{0}{0}},\infty)]\!]_{0,\square}\}$, then $a\mathbin{\oplus_{n,\square}}[\![(\mathord{\stackrel{0}{0}},\infty)]\!]_{0,\square}=[\![(\mathord{\stackrel{0}{0}},\infty)]\!]_{0,\square}\mathbin{\oplus_{n,\square}}a=[\![(\mathord{\stackrel{0}{0}},\infty)]\!]_{0,\square}$.
\item If $a\in\mathbb{Z}_{0,\square}\cup\{[\![(\infty,\mathord{\stackrel{0}{0}})]\!]_{0,\square}\}$, then $a\mathbin{\oplus_{n,\square}}[\![(\infty,\mathord{\stackrel{0}{0}})]\!]_{0,\square}=[\![(\infty,\mathord{\stackrel{0}{0}})]\!]_{0,\square}\mathbin{\oplus_{n,\square}}a=[\![(\infty,\mathord{\stackrel{0}{0}})]\!]_{0,\square}$.
\end{enumerate}
There are no inverse elements for $[\![(\mathord{\stackrel{0}{0}},\infty)]\!]_{0,\square}$ and $[\![(\infty,\mathord{\stackrel{0}{0}})]\!]_{0,\square}$ because the expressions $[\![(\mathord{\stackrel{0}{0}},\infty)]\!]_{0,\square}\mathbin{\oplus_{n,\square}}[\![(\infty,\mathord{\stackrel{0}{0}})]\!]_{0,\square}$ and $[\![(\infty,\mathord{\stackrel{0}{0}})]\!]_{0,\square}\mathbin{\oplus_{n,\square}}[\![(\mathord{\stackrel{0}{0}},\infty)]\!]_{0,\square}$ are undefined.  The tuple $(\mathbb{Z}_{0,\square}\cup\{[\![(\mathord{\stackrel{0}{0}},\infty)]\!]_{0,\square},[\![(\infty,\mathord{\stackrel{0}{0}})]\!]_{0,\square}\},\oplus_{0,\square})$ is therefore a monoid, but not a group.  It contains the submonoids $(\{[\![(\mathord{\stackrel{0}{0}},\infty)]\!]_{0,\square}\}\cup T_{0,\square}\:\mathbb{N}_{0,\square},\oplus_{0,\square})$ and $(\mathbb{N}_{0,\square}\cup\{[\![(\infty,\mathord{\stackrel{0}{0}})]\!]_{0,\square}\},\oplus_{0,\square})$, related by an involution $\mathcal{T}_{0,\square}$, such that
\begin{enumerate}
\item If $a\in\mathbb{Z}_{0,\square}$, then $\mathcal{T}_{0,\square}\:a=T_{0,\square}\:a$.
\item $\mathcal{T}_{0,\square}\:[\![(\mathord{\stackrel{0}{0}},\infty)]\!]_{0,\square}=[\![(\infty,\mathord{\stackrel{0}{0}})]\!]_{0,\square}$.
\item $\mathcal{T}_{0,\square}\:[\![(\infty,\mathord{\stackrel{0}{0}})]\!]_{0,\square}=[\![(\mathord{\stackrel{0}{0}},\infty)]\!]_{0,\square}$.
\end{enumerate}  
The change in notation from $T_{n,\square}$ to $\mathcal{T}_{0,\square}$ is to emphasize that the latter is no longer a group inverse operation. Extend the domain of the function $E$ from $\mathord{\stackrel{0}{\mathbb{N}}}$ to $\mathord{\stackrel{0}{\mathbb{N}}}\cup\{\infty\}$.  From the finite-base exponentiation of the \textit{extended real number line} one obtains \cite{BO07}
\begin{equation}\label{eq:baki1}
\exp_{\omega}(\psi(\infty))=\psi(\infty).
\end{equation}
Note that $\psi(\infty)$ fulfills the same property as $h(x)\!\restriction_{x=w_{0}}$, defined in Eq.~\eqref{eq:inftower}.  Since $\exp_{\omega}(\psi(\infty))=\psi(E(\infty))$, consistency requires
\begin{equation}
E(\infty)=\infty.
\end{equation}
Henceforth assumed.  A similar construction to the one above, leads to the monoid $(\mathbb{Z}_{n,\square}\cup\{[\![(\mathord{\stackrel{n}{0}},\infty)]\!]_{n,\square},[\![(\infty,\mathord{\stackrel{n}{0}})]\!]_{n,\square}\},\oplus_{n,\square})$.  It contains the submonoids $(\{[\![(\mathord{\stackrel{n}{0}},\infty)]\!]_{n,\square}\}\cup T_{n,\square}\:\mathbb{N}_{n,\square},\oplus_{n,\square})$ and $(\mathbb{N}_{n,\square}\cup\{[\![(\infty,\mathord{\stackrel{n}{0}})]\!]_{n,\square}\},\oplus_{n,\square})$, related by the involution $\mathcal{T}_{n,\square}$.  Extend the domain of the functions $E_{\square}$ and $\psi_{\square}$ from $U_{\mathbb{Z},\square}$ to $U_{\mathbb{Z},\square}\cup\{[\![(\mathord{\stackrel{n}{0}},\infty)]\!]_{n,\square}\}_{n\in\mathbb{N}}\cup\{[\![(\infty,\mathord{\stackrel{n}{0}})]\!]_{n,\square}\}_{n\in\mathbb{N}}$.  The Eq.~\eqref{eq:baki1} can also be written as
\begin{equation}\label{eq:baki2}
\exp_{\omega}(\psi_{\square}([\![(\infty,\mathord{\stackrel{0}{0}})]\!]_{0,\square}))=\psi_{\square}([\![(\infty,\mathord{\stackrel{0}{0}})]\!]_{0,\square}).
\end{equation}
Since
\begin{equation}
\exp_{\omega}(\psi_{\square}([\![(\infty,\mathord{\stackrel{0}{0}})]\!]_{0,\square}))=\psi_{\square}([\![(\infty,\mathord{\stackrel{1}{0}})]\!]_{1,\square}),
\end{equation}
the Eq.~\eqref{eq:baki2} becomes
\begin{equation}\label{eq:baki3}
\psi_{\square}([\![(\infty,\mathord{\stackrel{1}{0}})]\!]_{1,\square})=\psi_{\square}([\![(\infty,\mathord{\stackrel{0}{0}})]\!]_{0,\square}).
\end{equation}
Henceforth assumed.  Applying the function $\exp_{w_{0}}^{n}$ at both sides of Eq.~\eqref{eq:baki3} one obtains
\begin{equation}\label{eq:baki4}
\psi_{\square}([\![(\infty,\mathord{\stackrel{n+1}{0}})]\!]_{n+1,\square})=\psi_{\square}([\![(\infty,\mathord{\stackrel{n}{0}})]\!]_{n,\square}).
\end{equation}
From the algebraic properties of the extended real number line, one obtains
\begin{equation}\label{eq:enzo}
0_{0}=\psi_{\square}([\![(\mathord{\stackrel{0}{0}},\mathord{\stackrel{0}{0}})]\!]_{0,\square})=\psi_{\square}([\![(\mathord{\stackrel{1}{0}},\infty)]\!]_{1,\square}).
\end{equation}
Applying the function $\exp_{w_{0}}^{n}$ at both sides, one obtains
\begin{equation}
0_{n}=\psi_{\square}([\![(\mathord{\stackrel{n}{0}},\mathord{\stackrel{n}{0}})]\!]_{n,\square})=\psi_{\square}([\![(\mathord{\stackrel{n+1}{0}},\infty)]\!]_{n+1,\square}).
\end{equation}
And applying the function $\log_{w_{0}}$ at both sides of Eq.~\eqref{eq:enzo} one obtains
\begin{equation}
0_{-1}=\psi_{\square}([\![(\mathord{\stackrel{-1}{0}},\mathord{\stackrel{-1}{0}})]\!]_{-1,\square})=\psi_{\square}([\![(\mathord{\stackrel{0}{0}},\infty)]\!]_{0,\square}).
\end{equation}
This exposes the difference between $\{[\![(\infty,\mathord{\stackrel{n}{0}})]\!]_{n,\square}\}_{n\in\mathbb{N}}$ and $\{[\![(\mathord{\stackrel{n}{0}},\infty)]\!]_{n,\square}\}_{n\in\mathbb{N}}$.  
It also implies
\begin{equation}
\psi_{\square}(U_{\mathbb{Z},\square}\cup\{[\![(\mathord{\stackrel{n}{0}},\infty)]\!]_{n,\square}\}_{n\in\mathbb{N}}\cup\{[\![(\infty,\mathord{\stackrel{n}{0}})]\!]_{n,\square}\}_{n\in\mathbb{N}}
)=U_{\mathbb{Z}}\cup\{0_{-1},\psi(\infty)\}.
\end{equation}
In particular, note that
\begin{equation}
\psi_{\square}(\mathbb{Z}_{n,\square}\cup \{[\![(\mathord{\stackrel{n}{0}},\infty)]\!]_{n,\square},[\![(\infty,\mathord{\stackrel{n}{0}})]\!]_{n,\square}\})=\mathbb{Z}_{n}\cup\{0_{n-1},\psi(\infty)\}.
\end{equation}
Denote by $\psi_{n,\square}$, the restriction of $\psi_{\square}$ to the subdomain $\mathbb{Z}_{n,\square}\cup \{[\![(\mathord{\stackrel{n}{0}},\infty)]\!]_{n,\square},[\![(\infty,\mathord{\stackrel{n}{0}})]\!]_{n,\square}\}$.  Define the involution $\mathcal{T}_{n}$ in $\mathbb{Z}_{n}\cup\{0_{n-1},\psi(\infty)\}$, such that  
\begin{equation}
\mathcal{T}_{n}\:\psi_{\square}(a)=\psi_{\square}(\mathcal{T}_{n,\square}\:a).
\end{equation}
Extend the binary operation $\bullet_{n}$ to the domain $\mathbb{R}_{n}\cup\{0_{n-1},\psi(\infty)\}$, such that
\begin{enumerate}
\item If $a\in\mathbb{Z}_{n}\cup\{0_{n-1}\}$, then $a\bullet_{n}0_{n-1}=0_{n-1}\bullet_{n}a=0_{n-1}$.
\item If $a\in\mathbb{Z}_{n}\cup\{\psi(\infty)\}$, then $a\bullet_{n}\psi(\infty)=\psi(\infty)\bullet_{n}a=\psi(\infty)$.
\end{enumerate} 
The involution $\mathcal{T}_{n}$ is an isomorphism between the monoids $(\{0_{n-1}\}\cup\mathcal{T}_{n}\:\mathbb{N}_{n},\bullet_{n})$ and $(\mathbb{N}_{n}\cup\{\psi(\infty)\},\bullet_{n})$, in particular
\begin{equation}
\mathcal{T}_{n}\:\psi(\infty)=0_{n-1}.
\end{equation}
\end{remark}
\begin{remark}
Recall the tuple $(\mathord{\stackrel{0}{\mathbb{N}}},F_{0},F_{1})$ is the semiring of natural numbers.  Consider its completion as a new semiring where the limit $\lim_{w_{0}\rightarrow\infty}E(\mathord{\stackrel{0}{\mathbb{N}}})$ is defined.  Denote the latter limit as $R_{1}$ and $R_{0}=\mathord{\stackrel{0}{\mathbb{N}}}$.  Denote the function $\mathcal{E}(.)=\lim_{w_{0}\rightarrow\infty}E(.)$ in $R_{0}$, with inverse $\mathcal{L}(.)$.  Note that $R_{0}\cap R_{1}=\{\mathcal{E}(0)=1\}$.  Denote $R_{0}^{\ast}=R_{0}\setminus\{0\}$ and $R_{1}^{\ast}\setminus\{\mathcal{E}(0)\}$.  Let $a,b\in R_{0}$ and assume
\begin{equation}
\mathcal{E}(a)\leq\mathcal{E}(b)\iff a\leq b.
\end{equation}
This contrasts with the usual extended natural numbers, where there is only one nonfinite number, which remains invariant when elevated to finite exponents \cite{BO07}.  Define the operations $\mathcal{F}_{0}$ and $\mathcal{F}_{1}$, such that
\begin{enumerate}
\item If $a,b\in R_{0}$, then $\mathcal{F}_{0}(a,b)=F_{0}(a,b)$.
\item If $a,b\in R_{1}^{\ast}$, then $\mathcal{F}_{0}(a,b)
=\text{max}(a,b)$.
\item If $a\in R_{0}$ and $b\in R_{1}^{\ast}$, then $\mathcal{F}_{0}(a,b)
=b$.  
\item If $a,b\in R_{0}$, then $\mathcal{F}_{1}(a,b)=F_{1}(a,b)$.
\item If $a,b\in R_{1}^{\ast}$, then $\mathcal{F}_{1}(a,b)
=\mathcal{E}(\mathcal{F}_{0}(\mathcal{L}(a),\mathcal{L}(b))$.
\item If $a\in R_{0}$ and $b\in R_{1}^{\ast}$, then $\mathcal{F}_{1}(a,b)
=b$. 
\end{enumerate}
The tuple $(R_{0}\cup R_{1},\mathcal{F}_{0},\mathcal{F}_{1})$ is a semiring. 
Now extend recursively the latter semiring such that the limit $\lim_{w_{0}\rightarrow\infty}E(\mathord{\stackrel{n}{\mathbb{N}}})$ is defined for $n\in\mathbb{N}$ and denote each limit as $R_{n}$.  Consequently, extend the domain of $\mathcal{E}(.)$ to $\bigcup_{n\in\mathbb{N}}R_{n}$.  Note that $R_{n}\cap R_{n+1}=\{\mathcal{E}^{n+1}(0)=\mathcal{E}^{n}(1)\}$.  Denote $R_{n}^{\ast}=R_{n}\setminus\{\mathcal{E}^{n}(0)\}$.  Let $a,b\in R_{n}$ and assume
\begin{equation}
\mathcal{E}(a)=\mathcal{E}(b)\iff a=b.
\end{equation}
Define the operations $\mathcal{F}_{n}$ and $\mathcal{F}_{n+1}$ for $n\in\mathbb{N}\setminus\{0\}$, such that
\begin{enumerate}
\item If $a,b\in R_{n}$, then $\mathcal{F}_{n}(a,b)=\mathcal{E}(\mathcal{F}_{n-1}(\mathcal{L}(a),\mathcal{L}(b)))$.
\item If $a,b\in R_{n+1}^{\ast}$, then $\mathcal{F}_{n}(a,b)
=\text{max}(a,b)$.
\item If $a\in R_{n}$ and $b\in R_{n+1}^{\ast}$, then $\mathcal{F}_{n}(a,b)
=b$.
\item If $a,b\in R_{n}$, then $\mathcal{F}_{n+1}(a,b)
=\mathcal{E}(\mathcal{F}_{n}(\mathcal{L}(a),\mathcal{L}(b)))$.
\item If $a,b\in R_{n+1}^{\ast}$, then $\mathcal{F}_{n+1}(a,b)
=\mathcal{E}(\mathcal{F}_{n}(\mathcal{L}(a),\mathcal{L}(b)))$.
\item If $a\in R_{n}$ and $b\in R_{n+1}^{\ast}$, then $\mathcal{F}_{n+1}(a,b)
=b$. 
\end{enumerate}
The tuple $(R_{n}\cup R_{n+1},\mathcal{F}_{n},\mathcal{F}_{n+1})$ is a semiring.  The subset $R_{n+1}^{\ast}$ is an ideal, as it fulfills the properties \cite{GO99}:
\begin{enumerate}
\item The tuple $(R_{n+1}^{\ast},\mathcal{F}_{n})$ is a semigroup.
\item If $a\in R_{n}\cup R_{n+1}$ and $b\in R_{n+1}^{\ast}$, then $\mathcal{F}_{n+1}(a,b)\in I$.
\item $R_{n+1}^{\ast}\subset R_{n}\cup R_{n+1}$.
\end{enumerate}
\end{remark}

%% file: ms.bbl
\begin{thebibliography}{20}

 \bibitem{BE15}
 Bennett, A. A.
 ``Note on an Operation of the Third Grade.''
 \textit{Ann. of Math.}, Second Series, 17, no.~2 (1915): 74--75.
 \verb|doi:10.2307/2007124.|

 
 \bibitem{ACK28}
 Ackermann, W.
 ``Zum Hilbertschen Aufbau der reellen Zahlen.''
 \textit{Math. Ann.} 99 (1928): 118--133.
 \verb|doi:10.1007/BF01459088.|

 \bibitem{GO47}
 Goodstein, R. L.
 ``Transfinite ordinals in recursive number theory.''
 \textit{J. Symb. Log.} 12, no.~4 (1947): 123--129.
 \verb|doi:10.2307/2266486.|
 
 \bibitem{TAR69}
 Donner, J., and A. Tarski.
 ``An extended arithmetic of ordinal numbers.''
 \textit{Fund. Math.} 65 (1969): 95--127.
 \verb|doi:10.4064/fm-65-1-95-127.|
 
  \bibitem{SA07}
 Salomon, D.
 \textit{Variable-Length Codes for Data Compression}, 1st ed.
 London: Springer-Verlag, 2007.
 \verb|doi:10.1007/978-1-84628-959-0.|
 
  \bibitem{LI18}
 Lindstrom, P., S. Lloyd and J. Hittinger.
 ``Universal coding of the reals: Alternatives to IEEE floating point.''
 \textit{Conference for the Next Generation Arithmetic (ConNGA'18).  Association for Computing Machinery, New York, NY, USA} 5 (2018): 1--14.
 \verb|doi:10.1145/3190339.3190344.|
 
 \bibitem{EL75} 
 Elias, P.
 ``Universal codewords sets and representations of the integers.''
 \textit{IEEE Trans. Inform. Theory} 21, no.~2 (1975): 194--203.
 \verb|doi: 10.1109/TIT.1975.1055349.|
 
 \bibitem{CLE84}
 Clenshaw, C. W., and F. W. J. Olver.
 ``Beyond floating point.''
 \textit{J. ACM} 31, no.~2 (1984): 319--328.
 \verb|doi:10.1145/62.322429.|
 
 \bibitem{CLE88}
 Clenshaw, C. W., and P. R. Turner.
 ``The Symmetric Level-Index System.''
 \textit{IMA J. Numer. Anal.} 8, no.~4 (1988): 517--526.
 \verb|doi:10.1093/imanum/8.4.517.|

 \bibitem{SA09}
 Sakarovitch, J.
 \textit{Elements of Automata Theory}, 1st ed.
 Cambridge: Cambridge University Press, 2009.
 \verb|doi:10.1017/CBO9781139195218.|

 \bibitem{DR09}
 Manfred, D., K. Werner and V. Heiko.
 \textit{Handbook of Weighted Automata.}
 Monographs in Theoretical Computer Science.  An EATCS Series.
 Berlin Heidelberg: Springer-Verlag, 2009.
 \verb|doi:10.1007/978-3-642-01492-5.|


 \bibitem{WA65}
 Warner, S.
 \textit{Modern Algebra.}
 Dover Publications Inc., 1965.

 \bibitem{BO07}
 Bourbaki, N.
 \textit{Topologie générale. Chapitres 1 à 4.}  Berlin Heidelberg: Springer-Verlag, 2007. 
 \verb|doi:10.1007/978-3-540-33982-3|.

 \bibitem{GO99}
 Golan, J. S.
 \textit{Ideals in Semirings.}  In: Semrings and their Applications.
 Dordrecht: Springer, 1999.
 \verb|doi:10.1007/978-94-015-9333-5_6|.
 
\end{thebibliography}
